\documentclass[a4paper,11pt]{article}
\usepackage{amsmath,amsthm}
\usepackage{changepage}
\usepackage{authblk}
\usepackage[style=numeric-comp,maxbibnames=99,bibencoding=utf8,giveninits=true]{biblatex}
\usepackage{cancel}
\usepackage[nomarkers,nolists]{endfloat}
\usepackage[hmargin={25mm,25mm},vmargin={30mm,35mm}]{geometry}
\usepackage{loglogslopetriangle}
\usepackage{nicefrac}
\usepackage[colorlinks, allcolors=blue]{hyperref}
\usepackage{pgfplots}
\pgfplotsset{compat=1.18}
\usepackage{subcaption}
\usepackage[compact,small]{titlesec}
\usepackage{tikz}
\usepackage{newtxmath}
\usepackage{newtxtext}
\usepackage{aligned-overset}
\usepackage{enumitem}

\usetikzlibrary{decorations.markings}
\bibliography{curlcurl2d}

\newcommand{\email}[1]{\href{mailto:#1}{#1}}

\allowdisplaybreaks

\numberwithin{equation}{section}

\usepackage[normalem]{ulem}
\newcounter{corr}
\definecolor{violet}{rgb}{0.580,0.,0.827}
\newcommand{\corr}[3]{\typeout{Warning : a correction remains in page \thepage}
  \stepcounter{corr}        
	            {\color{blue}\ifmmode\text{\,\sout{\ensuremath{#1}}\,}\else\sout{#1}\fi}
              {\color{red}#2}
              {\color{violet} #3}
}

\newtheorem{theorem}{Theorem}

\newtheorem{lemma}[theorem]{Lemma}
\newtheorem{corollary}[theorem]{Corollary}
\theoremstyle{remark}
\newtheorem{remark}[theorem]{Remark}
\theoremstyle{definition}
\newtheorem{assumption}[theorem]{Assumption}
\newtheorem{definition}[theorem]{Definition}

\newcommand{\st}{\,:\,}
\newcommand{\Real}{\mathbb{R}}

\DeclareRobustCommand{\bvec}[1]{\boldsymbol{#1}}
\newcommand{\Id}{{\rm Id}}
\let\ser\widehat
\newcommand{\svec}[1]{\ser{\bvec{#1}}}

\newcommand{\uvec}[1]{\underline{\bvec{#1}}}
\newcommand{\suvec}[1]{\underline{\ser{\bvec{#1}}}}
\newcommand{\cvec}[1]{\bvec{\mathcal{#1}}}

\DeclareMathOperator{\GRAD}{\bf grad}
\DeclareMathOperator{\DIV}{div}
\DeclareMathOperator{\ROT}{rot}
\DeclareMathOperator{\VROT}{\bf rot}
\DeclareMathOperator{\CURL}{\bf curl}

\newcommand{\Hrotrot}[1]{\bvec{H}(\VROT\ROT;#1)}
\newcommand{\Hrot}[1]{\bvec{H}(\ROT;#1)}
\newcommand{\Hdiv}[1]{\bvec{H}(\DIV;#1)}

\newcommand{\Hcurl}[1]{\bvec{H}(\CURL;#1)}

\newcommand{\compl}{{\rm c}}

\newcommand{\Poly}[1]{\mathcal{P}^{#1}}
\newcommand{\vPoly}[1]{\cvec{P}^{#1}}
\newcommand{\Roly}[1]{\cvec{R}^{#1}}
\newcommand{\cRoly}[1]{\cvec{R}^{\compl,#1}}
\newcommand{\Goly}[1]{\cvec{G}^{#1}}
\newcommand{\cGoly}[1]{\cvec{G}^{\compl,#1}}
\newcommand{\TPoly}[1]{\boldsymbol{\mathcal{P}}_{\normal}^{#1}(T_1)}

\newcommand{\lPoly}[2][]{\mathcal{P}_{#1}^{#2}}

\newcommand{\lproj}[2]{\pi_{\mathcal{P},#2}^{#1}}
\newcommand{\vlproj}[2]{\bvec{\pi}_{\cvec{P},#2}^{#1}}
\newcommand{\Rproj}[2]{\bvec{\pi}_{\cvec{R},#2}^{#1}}
\newcommand{\cRproj}[2]{\bvec{\pi}_{\cvec{R},#2}^{\compl,#1}}
\newcommand{\Gproj}[2]{\bvec{\pi}_{\cvec{G},#2}^{#1}}
\newcommand{\Gcproj}[2]{\bvec{\pi}_{\cvec{G},#2}^{\compl,#1}}
\newcommand{\Xcproj}[2]{\bvec{\pi}_{\cvec{X},#2}^{\compl,#1}}

\newcommand{\normal}{\bvec{n}}
\newcommand{\tangent}{\bvec{t}}

\DeclareMathOperator{\Kernel}{Ker}
\DeclareMathOperator{\Image}{Im}

\newcommand{\norm}[2]{\|#2\|_{#1}}

\newcommand{\vvvert}{\vert\kern-0.25ex\vert\kern-0.25ex\vert}

\newcommand{\Xgrad}[2]{\underline{W}_{\GRAD,#2}^{#1}}
\newcommand{\Xcurl}[2]{\uvec{W}_{\CURL,#2}^{#1}}
\newcommand{\Xrot}[2]{\uvec{W}_{\ROT,#2}^{#1}}
\newcommand{\Xdiv}[2][k]{\uvec{W}_{\DIV,#2}^{#1}}

\newcommand{\XL}[2]{W_{L^2,#2}^{#1}}
\newcommand{\SXgrad}[2]{\underline{\widehat{W}}_{\GRAD,#2}^{#1}}

\newcommand{\SV}[1]{\underline{V}_{\GRAD,#1}^k}
\newcommand{\SSigma}[1]{\uvec{V}_{\ROT,#1}^k}
\newcommand{\SW}[1]{\underline{V}_{H^1,#1}^k}
\newcommand{\SGr}[1]{\underline{V}_{\GRAD,#1}^k}
\newcommand{\SCr}[1]{\uvec{V}_{\CURL,#1}^k}
\newcommand{\SDr}[1]{\uvec{V}_{\DIV,#1}^k}
\newcommand{\SVser}[1]{\widehat{\underline{V}}_{\GRAD,#1}^k}
\newcommand{\SSigmaser}[1]{\suvec{V}_{\ROT,#1}^k}
\newcommand{\SXrot}[1]{\suvec{W}_{\ROT,#1}^k}
\newcommand{\SXcurl}[1]{\suvec{W}_{\CURL,#1}^k}

\newcommand{\SSGr}[1]{\widehat{\underline{V}}_{\GRAD,#1}^k}
\newcommand{\SSCr}[1]{\suvec{V}_{\CURL,#1}^k}

\newcommand{\Injrot}{\uvec{\mathcal E}_{\ROT,h}^k}
\newcommand{\Resrot}{\uvec{\mathcal R}_{\ROT,h}^k}

\newcommand{\sInjrot}{\suvec{\mathcal E}_{\ROT,h}^k}
\newcommand{\sResrot}{\suvec{\mathcal R}_{\ROT,h}^k}

\newcommand{\ResW}[1]{\mathcal {R}_{H^1,#1}^k}
\newcommand{\InjW}[1]{\underline{\mathcal E}_{H^1,#1}^k}
\newcommand{\Injgrad}{\underline{\mathcal E}_{\GRAD,h}^k}
\newcommand{\Injcurl}{\uvec{\mathcal E}_{\CURL,h}^k}
\newcommand{\Injdiv}{\uvec{\mathcal E}_{\DIV,h}^k}
\newcommand{\Resgrad}{\underline{\mathcal R}_{\GRAD,h}^k}
\newcommand{\Rescurl}{\uvec{\mathcal R}_{\CURL,h}^k}
\newcommand{\Resdiv}{\uvec{\mathcal R}_{\DIV,h}^k}
\newcommand{\sInjgrad}{\widehat{\underline{\mathcal E}}_{\GRAD,h}^k}
\newcommand{\sInjcurl}{\suvec{\mathcal E}_{\CURL,h}^k}
\newcommand{\sResgrad}{\widehat{\underline{\mathcal R}}_{\GRAD,h}^k}
\newcommand{\sRescurl}{\suvec{\mathcal R}_{\CURL,h}^k}

\newcommand{\sIrot}[1]{\suvec{I}_{\ROT,#1}^{k}}
\newcommand{\sIcurl}[1]{\suvec{I}_{\CURL,#1}^{k}}
\newcommand{\sIGrad}[1]{{\underline{\widehat I}}_{\GRAD,#1}^{k}}
\newcommand{\IV}[1]{\underline{I}_{V,#1}^{k}}
\newcommand{\sIV}[1]{\underline{\widehat I}_{V,#1}^{k}}
\newcommand{\ISigma}[1]{\uvec{I}_{\bvec{\Sigma},#1}^k}
\newcommand{\sISigma}[1]{\suvec{I}_{\bvec{\Sigma},#1}^k}
\newcommand{\GE}{G_{T_1}^k}
\newcommand{\GT}{\bvec{G}_{T_3}^k}
\newcommand{\GF}{\bvec{G}_{T_2}^k}

\newcommand{\Gddr}[1]{\uvec{\partial}_{\GRAD,#1}^k}
\newcommand{\rotddr}[1]{\partial_{\ROT,#1}^k}
\newcommand{\sGddr}[1]{\suvec{\partial}_{\GRAD,#1}^{k}}
\newcommand{\sCddr}[1]{\suvec{\partial}_{\CURL,#1}^{k}}

\newcommand{\sRddr}[1]{\widehat{\partial}_{\ROT,#1}^{k}}
\newcommand{\Cddr}[1]{\uvec{\partial}_{\CURL,#1}^{k}}
\newcommand{\Ext}[1]{\mathcal{E}_{#1}}
\newcommand{\Red}[1]{\mathcal{R}_{#1}}
\newcommand{\sExt}[1]{\widehat{\mathcal{E}}_{#1}}
\newcommand{\sRed}[1]{\widehat{\mathcal{R}}_{#1}}
\newcommand{\ReW}[2]{\widehat{R}_{{#1}_{#2}}}
\newcommand{\EX}[2]{E_{{#1}_{#2}}}
\newcommand{\EPoly}[2][k-1]{E_{\Poly{},#2}^{#1}}
\newcommand{\RPoly}[1]{\widehat{R}^{\ell_{#1}}_{\Poly{},#1}}
\newcommand{\RRoly}[1]{\svec{R}^{k-1}_{\cvec{R},#1}}
\newcommand{\trF}{\gamma_{T_2}^{k+1}}
\newcommand{\trFt}{\bvec{\gamma}_{{\rm t},T_2}^k}
\newcommand{\trE}{\gamma_{T_1}^{k+1}}

\newcommand{\uGh}[1]{\uvec{d}_{\GRAD,{#1}}^k}
\newcommand{\uCh}[1]{\uvec{d}_{\CURL,{#1}}^k}

\newcommand{\uRh}[1]{\underline{d}_{\ROT,#1}^k}
\newcommand{\CF}[1][k]{C_{T_2}^{#1}}
\newcommand{\CT}[1][k]{\bvec{C}_{T_3}^{#1}}
\newcommand{\Dh}[1]{\partial_{\DIV,#1}^k}
\newcommand{\DT}[1][k]{D_{T_3}^{#1}}
\newcommand{\suRh}[1]{\underline{\widehat{d}}_{\ROT,#1}^k}
\newcommand{\uDh}[1]{d_{\DIV,#1}^k}

\newcommand{\SerGrad}[1]{\bvec{S}_{\GRAD,#1}^k}

\newcommand{\SerCurl}[1]{\bvec{S}_{\CURL,#1}^k}
\newcommand{\SerRot}[1]{\bvec{S}_{\ROT,#1}^k}

\newcommand{\suGh}[1]{\suvec{d}_{\GRAD,#1}^k}
\newcommand{\suCh}[1]{\suvec{d}_{\CURL,#1}^k}

\newcommand{\RV}[1]{\underline{\widehat{R}}_{V,\GRAD,#1}}
\newcommand{\EV}[1]{\underline{E}_{V,\GRAD,#1}}
\newcommand{\Rrot}[1]{\suvec{R}_{\bvec{W},\ROT,#1}}
\newcommand{\Erot}[1]{\uvec{E}_{\bvec{W},\ROT,#1}}
\newcommand{\Rcurl}[1]{\suvec{R}_{\bvec{W},\CURL,#1}}
\newcommand{\Ecurl}[1]{\uvec{E}_{\bvec{W},\CURL,#1}}
\newcommand{\Egrad}[1]{\underline{E}_{W,\GRAD,#1}}
\newcommand{\Rgrad}[1]{\underline{\widehat{R}}_{\bvec{W},\GRAD,#1}}
\newcommand{\RS}[1]{\suvec{R}_{\bvec{V},\ROT,#1}}
\newcommand{\ES}[1]{\uvec{E}_{\bvec{V},\ROT,#1}}

\begin{document}

\title{Serendipity discrete complexes with enhanced regularity}
\author[1]{Daniele A. Di Pietro}
\author[2]{Marien Hanot}
\author[1]{Marwa Salah}
\affil[1]{IMAG, Univ Montpellier, CNRS, Montpellier, France, \email{daniele.di-pietro@umontpellier.fr}, \email{marwa.salah@umontpellier.fr}}
\affil[2]{University of Edinburgh, Edinburgh, United Kingdom, \email{mhanot@ed.ac.uk}}

\maketitle

\begin{abstract}
  In this work we address the problem of finding serendipity versions of approximate de Rham complexes with enhanced regularity.
  The starting point is a new abstract construction of general scope which, given three complexes linked by extension and reduction maps, generates a fourth complex with cohomology isomorphic to the former three.
  This construction is used to devise new serendipity versions of rot-rot and Stokes complexes derived in the Discrete de Rham spirit.
\end{abstract}



\section{Introduction}

In this work we address the question of finding serendipity versions of discrete de Rham complexes with enhanced regularity.
The starting point is a new construction of general scope which, given three complexes connected by extension and reduction maps in the spirit of \cite{Di-Pietro.Droniou:23*1}, generates a fourth complex with cohomology isomorphic to the former three.

In the context of finite elements, the word ``serendipity'' refers to the possibility, on certain element geometries, to discard some internal degrees of freedom (DOFs) without modifying the approximation properties of the underlying space; see, e.g., \cite{Arnold.Awanou:11,Gillette.Klofkorn:19} for recent developments in the context of the approximation of Hilbert complexes.
In the context of arbitrary-order polyhedral methods, serendipity techniques were first developed in \cite{Beirao-da-Veiga.Brezzi.ea:18} to build a reduced version of the nodal ($H^1$-conforming) virtual space.
Similar ideas had been previously followed in \cite{Di-Pietro:12} to reduce the number of element DOFs in the framework of discontinuous Galerkin methods and in \cite{Eymard.Gallouet.ea:10} to eliminate element DOFs in hybrid finite volume methods; see also \cite{Droniou.Eymard.ea:18} on this subject.

When applying serendipity techniques to a discrete complex rather than a single space, one must make sure that the elimination of DOFs does not alter its homological properties.
Compatible serendipity techniques to reduce the number of face DOFs in virtual element discretizations of the de Rham complex have been developed in \cite{Beirao-da-Veiga.Brezzi.ea:17,Beirao-da-Veiga.Brezzi.ea:18*1}, where a direct proof of local exactness properties was provided.
A variation of the discrete complex in the previous reference has been recently proposed in \cite{Beirao-da-Veiga.Dassi.ea:22}, where links with Discrete de Rham (DDR) methods have also been established.
A systematic approach to serendipity for polyhedral approximations of discrete complexes, including the elimination of both element and face DOFs, has been recently proposed in \cite{Di-Pietro.Droniou:23*1} and applied to the DDR complex of \cite{Di-Pietro.Droniou:23*2} (see also \cite{Di-Pietro.Droniou.ea:20} for preliminary developments and \cite{Bonaldi.Di-Pietro.ea:23} for an extension to differential forms).

In practical applications, Hilbert complexes different from (but typically linked to \cite{Arnold.Hu:21}) the de Rham complex are often relevant.
Examples include:
the rot-rot complex, which naturally arises when considering quad-rot problems;
the Stokes complex, relevant for incompressible flow problems;
the div-div complex, appearing in the modeling of thin plates.
Discretizations of such complexes in the DDR spirit have been recently proposed in \cite{Di-Pietro:23}, \cite{Hanot:23}, and \cite{Di-Pietro.Droniou:23}, respectively.
To this date, however, the literature on serendipity techniques for advanced Hilbert complexes is extremely limited.
An example in the context of polyhedral methods is provided by \cite{Botti.Di-Pietro.ea:23}, where a serendipity version of the DDR div-div complex is proposed and studied.
The goal of the present work is to fill this gap by proposing a general construction that makes it possible to derive in a systematic way a serendipity version of an advanced discrete complex whenever a serendipity version of the underlying de Rham complex is available.
The construction is applied to the derivation and study of discrete versions of the discrete rot-rot and Stokes complexes of \cite{Di-Pietro:23,Hanot:23}.

The rest of this work is organized as follows.
In Section~\ref{sec:gen} we present the abstract construction.
The discrete de Rham complex of \cite{Di-Pietro.Droniou:23*2} along with its serendipity version of \cite{Di-Pietro.Droniou:23*1} are briefly recalled in Section~\ref{sec:ddr.sddr}.
Serendipity versions of the rot-rot complex of \cite{Di-Pietro:23} and of the Stokes complex of \cite{Hanot:23} are derived and studied in Section~\ref{sec:rotrot} and \ref{sec:stokes}, respectively.
Section~\ref{sec:rotrot} also contains numerical experiments comparing the performance of the serendipity and original rot-rot complexes on a quad-rot problem.  


\section{An abstract framework for serendipity complexes with enhanced regularity}\label{sec:gen}

In this section we present an abstract framework that, given three complexes linked by suitable reduction and extension operators, allows one to construct a fourth complex with cohomology isomorphic to the others.
The application that we have in mind is the construction of serendipity versions of the de Rham complex with enhanced regularity.

\subsection{Setting}

We consider the situation depicted in the following diagram, involving three complexes $(W_i,\partial_i)_i$, $(\widehat{W}_i,\widehat{\partial_i})_i$, and $(V_i,d_i)_i$ :
\begin{equation}\label{eq:abstract:starting.point}
  \begin{tikzpicture}[xscale=2.5, yscale=1.75, baseline={(Wi.base)}]
    \node (SWi) at (0,1) {$\widehat{W}_i$};
    \node (SWi1) at (1,1) {$\widehat{W}_{i+1}$};

    \node (Wi) at (0,0) {$W_i$};
    \node (Wi1) at (1,0) {$W_{i+1}$};
    
    \node (Vi) at (0,-1) {$V_i$};
    \node (Vi1) at (1,-1) {$V_{i+1}$};

    \draw [->] (SWi) to node[above, font=\footnotesize]{$\widehat{\partial}_i$} (SWi1);
    \draw [->] (Wi) to node[above, font=\footnotesize]{$\partial_i$} (Wi1);
    \draw [->] (Vi) to node[above, font=\footnotesize]{$d_i$} (Vi1);
    
    \draw [->,dashed] (-1,-1) -- (Vi);
    \draw [->,dashed] (-1,0) -- (Wi);
    \draw [->,dashed] (-1,1) -- (SWi);

    \draw [->,dashed] (SWi1) -- (2,1);
    \draw [->,dashed] (Wi1) -- (2,0);
    \draw [->,dashed] (Vi1) -- (2,-1);

    \draw [->,dashed] (Wi) to [bend left=20] node[left, font=\footnotesize] {$\ReW{W}{i}$}(SWi);
    \draw [->,dashed] (Wi1) to [bend left=20] node[left, font=\footnotesize] {$\ReW{W}{i+1}$}(SWi1);

    \draw [->] (SWi) to [bend left=20] node[right, font=\footnotesize] {$\EX{W}{i}$}(Wi);
    \draw [->] (SWi1) to [bend left=20] node[right, font=\footnotesize] {$\EX{W}{i+1}$}(Wi1);

    \draw [->,dashed] (Vi) to [bend left=20] node[left, font=\footnotesize] {$\Red{i}$}(Wi);
    \draw [->,dashed] (Vi1) to [bend left=20] node[left, font=\footnotesize] {$\Red{i+1}$}(Wi1);

    \draw [->] (Wi) to [bend left=20] node[right, font=\footnotesize] {$\Ext{i}$}(Vi);
    \draw [->] (Wi1) to [bend left=20] node[right, font=\footnotesize] {$\Ext{i+1}$}(Vi1);
  \end{tikzpicture}
\end{equation}
The complexes $(W_i,\partial_i)_i$ and $(\widehat{W}_i,\widehat{\partial_i})_i$ are linked by linear extension and reduction operators $\EX{W}{i}\st \widehat{W}_i \to W_i$ and $\ReW{W}{i}\st W_i \to \widehat{W}_i$ that meet the following assumption.

\begin{assumption}[Properties of $\EX{W}{i}$ and $\ReW{W}{i}$]\label{Isom.W.W} 
    It holds:
  \begin{enumerate}[label=\textbf{(A{\arabic*})}]
  \item\label{P:Identity.ERW} $(\ReW{W}{i}\EX{W}{i})_{|\Kernel\widehat{\partial}_i} =\Id_{\Kernel\widehat{\partial}_i} $.
  \item\label{P:complex.ERW} $(\EX{W}{i+1}\ReW{W}{i+1}-\Id_{W{i+1}})(\Kernel \partial_{i+1}) \subset \Image(\partial_i)$.
  \item\label{P:E.R.W.cochain} $\ReW{W}{i+1} \partial_i=\widehat{\partial}_i\ReW{W}{i}$ and $\EX{W}{i+1}\widehat{\partial}_i= \partial_i \EX{W}{i}$.
  \end{enumerate}
\end{assumption}
By \cite[Proposition~2]{Di-Pietro.Droniou:23*1}, Assumption~\ref{Isom.W.W} guarantees the cohomologies of the complexes $(W_i,\partial_i)_i$ and $(\widehat{W}_i,\widehat{\partial_i})_i$ are isomorphic. 
Additionally, the upper diagram in \eqref{eq:abstract:starting.point} is commutative and we have:
\begin{equation}\label{eq:comm.diag}
  \widehat{\partial}_i=\ReW{W}{i+1}{\partial}_i\EX{W}{i}.
\end{equation}
Examples of complexes $(W_i,\partial_i)_i$ and $(\widehat{W}_i,\widehat{\partial_i})_i$ and of the corresponding reduction and extension operators that match Assumption~\ref{Isom.W.W} are provided by the two- and three-dimensional discrete de Rham complexes \eqref{eq:glob.2D-ddr.complex} and \eqref{eq:glob.3d-ddr.complex} below and their serendipity versions recalled in Section \ref{sec:ddr.sddr}.

In the applications of Sections~\ref{sec:rotrot} and \ref{sec:stokes} , $(V_i,d_i)_i$ is an extended version of $(W_i,\partial_i)_i$ with enhanced regularity, which is linked to $(W_i, \partial_i)_i$ by the linear extension and reduction operators $\Ext{i}\st W_i\to V_i$ and $\Red{i}\st V_i\to W_i$.

\begin{assumption}[Properties of $\Ext{i}$ and $\Red{i}$]\label{Isom.W.V}  
  It holds:
  \begin{enumerate}[label=\textbf{(B{\arabic*})}]
  \item\label{P:Identity.ER} $\Red{i}\Ext{i} =\Id_{W_{i}} $.
  \item\label{P:complex.ER} $(\Ext{i+1}\Red{i+1}-\Id_{V_{i+1}})(\Kernel d_{i+1}) \subset \Image(d_i)$.
  \item\label{P:E.R.cochain} $\Red{i+1} d_i=\partial_i\Red{i}$ and $\Ext{i+1}\partial_i= d_i \Ext{i}$.
  \end{enumerate}
\end{assumption}

\begin{remark}[Isomorphic cohomologies]\label{rem:iso.cohom}
  Notice that property \ref{P:Identity.ER} is stricter than \ref{P:Identity.ERW} since it requires $\Red{i}$ to be a left inverse of $\Ext{i}$ on the entire space $W_i$ and not only on $\Kernel \partial_i$.
  Accounting for this remark and invoking again \cite[Proposition~2]{Di-Pietro.Droniou:23*1}, it is easy to see that the cohomologies of $(V_i, d_i)_i$ and $(W_i, \partial_i)_i$ are isomorphic. As noticed above, the latter is, in turn, isomorphic to the cohomology of $(\widehat{W}_i,\widehat{\partial_i})_i$.
\end{remark}

The complex $(V_i, d_i)_i$ can be illustrated by the discrete rot-rot complex \eqref{eq:d.rotrot.complex} or the discrete Stokes complex \eqref{eq:d.stokes.complex}, respectively discussed in Sections~\ref{sec:rotrot} and~\ref{sec:stokes} below.

\begin{lemma}[Decomposition of $V_i$]\label{def.v}
  Assume \ref{P:Identity.ER} and let
  \begin{equation}\label{eq:Ci}
    C_i\coloneq\Kernel\Red{i}.
  \end{equation}
  Then, we have the following direct decomposition:
  \begin{equation}\label{eq:def.v}
    V_i = \Ext{i}W_i \oplus C_i.
  \end{equation}
  Under assumption \ref{P:E.R.cochain}, this decomposition is compatible with $d_i$, in the sense that
  \begin{equation}\label{eq:compatibility}
    \text{%
      $d_i \Ext{i}W_i \subset \Ext{i+1}W_{i+1}$ and
      $d_i C_i \subset C_{i+1}$.%
    }
  \end{equation}
\end{lemma}

\begin{proof}
  
  By \ref{P:Identity.ER}, $\Red{i}$ is surjective and $\Ext{i}$ is injective.
  As a consequence of the latter property, $|W_i|=|\Ext{i}W_i|$, where $| \cdot |$ denotes here the dimension of a vector space.
  By the rank-nullity theorem, we can also write $|C_i|= |V_i| - |\Image(\Red{i})| =|V_i|-|W_i|$, where the conclusion follows from the surjectivity of $\Red{i}$.
  Thus, $|C_i| + |\Ext{i}W_i|=|V_i|-|W_i|+|W_i|=|V_i|$, and this gives
  \[ 
  V_i=\Ext{i}W_i + C_i,
  \] 
    thus proving \eqref{eq:def.v}.
  
  Let us now prove that the sum in the above expression is direct. To this purpose, let $v\in \Ext{i}W_i \cap C_i$.
  Since $v \in C_i$, $\Red{i} v =0$.
  Since $v \in \Ext{i}W_i$, on the other hand, $v$ can be written as $\Ext{i}v_w$ for some $v_w \in W_i$, so $\Red{i}\Ext{i}v_w=0$.
  By \ref{P:Identity.ER}, $v_w=0$, so $v=\Ext{i}0=0$ (since $\Ext{i}$ is linear). As a result,
  \[ 
  \Ext{i}W_i \cap C_i=\{0\}.
  \] 
  
  Now, $\Red{i+1} d_i C_i \overset{\ref{P:E.R.cochain}}= \partial_i\Red{i} C_i \overset{\eqref{eq:Ci}}= 0$, giving that $d_i C_i \subset C_{i+1}$. On the other hand, $d_i \Ext{i} W_i \overset{\ref{P:E.R.cochain}} = \Ext{i+1} \partial_i W_i$, hence $d_i \Ext{i} W_i \subset \Ext{i+1} W_{i+1}$.
  This concludes the proof of \eqref{eq:compatibility}.
\end{proof}

\subsection{Construction of a serendipity complex with enhanced regularity}

The goal of this section is to construct a new complex $(\widehat{V}_i,\widehat{d}_i)$ with operators \[
\EX{V}{i}\st\widehat{V}_i\to V_i, \quad
\ReW{V}{i}\st V_i\to\widehat{V}_i, \quad
\sExt{i}\st\widehat{W}_i\to\widehat{V}_i, \quad
\sRed{i}\st\widehat{V}_i\to\widehat{W}_i,
\]
that verify conditions similar to the ones in Assumptions \ref{Isom.W.W} and \ref{Isom.W.V}, so that $(\widehat{V}_i,\widehat{d}_i)$ has the same cohomology as the three other complexes.
The construction is illustrated in the following diagram:
\begin{equation}\label{eq:abstract.diagram}
  \begin{tikzpicture}[xscale=2,yscale=2,baseline={(Middle.base)}]
    \node (SVi) at (0,0,0) {$\widehat{V}_i$};
    \node (SVi1) at (2,0,0) {$\widehat{V}_{i+1}$};

    \node (Vi) at (0,0,-2) {$V_i$};
    \node (Vi1) at (2,0,-2) {$V_{i+1}$};

    \node (SWi) at (0,2,0) {$\widehat{W}_i$};
    \node (SWi1) at (2,2,0) {$\widehat{W}_{i+1}$};

    \node (Wi) at (0,2,-2) {$W_i$};
    \node (Wi1) at (2,2,-2) {$W_{i+1}$};

    \node (Middle) at (0,1,0) {};

    \draw [->,>=latex] (SVi) -- (SVi1) node[midway, above, font=\small] {$\widehat{d}_i$};
    \draw [->,>=latex] (Vi) -- (Vi1) node[midway, above, font=\small] {$d_i$};
    \draw [->,>=latex] (SWi) -- (SWi1) node[midway, above, font=\small] {$\widehat{\partial}_i$};
    \draw [->,>=latex] (Wi) -- (Wi1) node[midway, above, font=\small] {$\partial_i$};

    \draw [->,>=latex] (SVi)to [bend right=20] node[pos=0.8, right, font=\small] {$\EX{V}{i}$} (Vi);
    \draw [->,dashed] (Vi)to [bend right=20] node[pos=0.5, left, font=\small] {$\ReW{V}{i}$} (SVi);
    \draw [->,>=latex] (SVi1)to [bend right=20] node[pos=0.7, right, font=\small] {$\EX{V}{i+1}$} (Vi1);
    \draw [->,dashed] (Vi1)to [bend right=20] node[pos=0.5, left, font=\small] {$\ReW{V}{i+1}$} (SVi1);

    \draw [->,>=latex] (SWi)to [bend right=20] node[pos=0.8, right, font=\small] {$\EX{W}{i}$} (Wi);
    \draw [->,dashed] (Wi)to [bend right=20] node[pos=0.5, left, font=\small] {$\ReW{W}{i}$} (SWi);
    \draw [->,>=latex] (SWi1)to [bend right=20] node[pos=0.7, right, font=\small] {$\EX{W}{i+1}$} (Wi1);
    \draw [->,dashed] (Wi1)to [bend right=20] node[pos=0.5, left, font=\small] {$\ReW{W}{i+1}$} (SWi1);

    \draw [->,>=latex] (Wi) to [bend left=15] node[pos=0.6, right,font=\small] {$\Ext{i}$} (Vi);
    \draw [->,dashed] (Vi) to [bend left=15] node[pos=0.4, left,font=\small] {$\Red{i}$} (Wi);

    \draw [->,>=latex] (Wi1) to [bend left=15] node[pos=0.6, right,font=\small] {$\Ext{i+1}$} (Vi1);
    \draw [->,dashed] (Vi1) to [bend left=15] node[pos=0.4, left,font=\small] {$\Red{i+1}$} (Wi1);

    \draw [->,>=latex] (SWi) to [bend left=15] node[pos=0.4, right,font=\small] {$\sExt{i}$} (SVi);
    \draw [->,dashed] (SVi) to [bend left=15] node[pos=0.6, left,font=\small] {$\sRed{i}$} (SWi);

    \draw [->,>=latex] (SWi1) to [bend left=15] node[pos=0.4, right,font=\small] {$\sExt{i+1}$} (SVi1);
    \draw [->,dashed] (SVi1) to [bend left=15] node[pos=0.6, left,font=\small] {$\sRed{i+1}$} (SWi1);
    \draw [->,dashed] (-1,0,0) -- (SVi);
    \draw [->,dashed] (-1,2,0) -- (SWi);
    \draw [->,dashed] (-1,0,-2) -- (Vi);
    \draw [->,dashed] (-1,2,-2) -- (Wi);
    \draw [->,dashed] (SVi1) -- (3,0,0);
    \draw [->,dashed] (SWi1) -- (3,2,0);
    \draw [->,dashed] (Vi1) -- (3,0,-2);
    \draw [->,dashed] (Wi1) -- (3,2,-2);
  \end{tikzpicture}
\end{equation}

By Lemma \ref{def.v}, a generic element $v \in V_i$ can be written as $v = \Ext{i}v_w + v_c$ with $(v_w, v_c) \in W_i \times C_i$. We introduce the projector $\Pi_{C_i}$ onto $C_i$  such that, for any 
$v = \Ext{i} v_w + v_c$,
\begin{equation}\label{eq:def.Pi}
  \Pi_{C_i} v \coloneq v_c.
\end{equation}
Notice that, by definition,
\begin{equation}\label{eq:ker.Pi}
  \Pi_{C_i} \Ext{i} = 0.
\end{equation}
In addition, using the compatibility expressed by \eqref{eq:compatibility}, 
\begin{equation}\label{eq:d.Pi}
  \Pi_{C_{i+1}} d_i v = d_i \Pi_{C_i}v,
\end{equation}
as can be checked writing
$\Pi_{C_{i+1}} d_i v = \Pi_{C_{i+1}} d_i (\Ext{i}v_w + v_c) = \Pi_{C_{i+1}} (d_i \Ext{i}v_w + d_i v_c) \overset{\eqref{eq:compatibility}}= d_i v_c \overset{\eqref{eq:def.Pi}}= d_i \Pi_{C_i}v $.

\begin{definition}[Complex $(\widehat{V}_i, \widehat{d}_i)$, extension and reduction operators]\label{def:serendipity.complex}
  The spaces and differential of the new complex are respectively given by
  \begin{equation}\label{eq:def.v^}
    \widehat{V}_i\coloneq \Big\{\widehat{v} = \big(\widehat{v}_{w},\widehat{v}_{c} \big)\st
    \text{$\widehat{v}_{w}\in\widehat{W}_i$ and $\widehat{v}_{c}\in C_i$}
    \Big\},
  \end{equation}
  and
  \begin{equation}\label{eq:def.d}
    \widehat{d}_i\widehat{v} \coloneq (\widehat{\partial}_i\widehat{v}_w, d_i\widehat{v}_c)~\text{for all} ~\widehat{v}= (\widehat{v}_w,\widehat{v}_c) \in \widehat{V}_i.
  \end{equation}
  The operators
  $\sExt{i}\st\widehat{W}_i\to\widehat{V}_i$,
  $\sRed{i}\st\widehat{V}_i\to\widehat{W}_i$,
  $\EX{V}{i}\st\widehat{V}_i\to V_i$, and
  $\ReW{V}{i}\st V_i\to\widehat{V}_i$
  relating this new complex to $(\widehat{W}_i,\widehat{\partial}_i)_i$ and $(V_i,d_i)_i$, respectively, are defined as follows:
  \begin{subequations}\label{eq:construction}
    \begin{alignat}{2}\label{eq:def.sExt}
      &\sExt{i} \widehat{v}_w &&\coloneq (\widehat{v}_w,0) ~ \text{for all}~ \widehat{v}_w \in \widehat{W}_i,\\ \label{eq:def.sRed}
      &\sRed{i} \widehat{v}&&\coloneq\widehat{v}_{w} ~\text{for all}~  \widehat{v} =(\widehat{v}_{w},\widehat{v}_{c}) \in \widehat{V}_i,\\ \label{eq:def.EX}
      &\EX{V}{i} \widehat{v} &&\coloneq \Ext{i}\EX{W}{i}\widehat{v}_w+\widehat{v}_c ~\text{for all}~ \widehat{v} =(\widehat{v}_{w},\widehat{v}_{c}) \in \widehat{V}_i,\\ \label{eq:def.ReV}
      &\ReW{V}{i} v&&\coloneq(\ReW{W}{i}\Red{i}v,\Pi_{C_i} v) ~\text{for all} ~v \in V_i.
    \end{alignat}
  \end{subequations}
\end{definition}

\begin{lemma}[Commutation properties]\label{comm.prop}
  Under Assumptions \ref{Isom.W.W} and \ref{Isom.W.V}, the operators defined by \eqref{eq:construction} satisfy the following relations:
  \begin{subequations}
    \begin{alignat}{2}\label{eq:RW.RV1}
      \widehat{R}_{W_i}\mathcal{R}_i
      &=
      \widehat{\mathcal{R}}_i\widehat{R}_{V_i},\\\label{eq:RW.RV2}
      \widehat{\mathcal{E}}_i \widehat{R}_{W_i}
      &=
      \widehat{R}_{V_i} \mathcal{E}_i,\\ \label{eq:EW.EV1}
      E_{W_i}\widehat{\mathcal{R}}_i
      &=
      \mathcal{R}_i E_{V_i}, \\ \label{eq:EW.EV2}
      \mathcal{E}_i E_{W_i}
      &=
      E_{V_i} \widehat{\mathcal{E}}_i, \\ \label{eq:part.d}
      \widehat{\partial_i}\sRed{i}
      &=
      \sRed{i+1}\widehat{d}_i.
    \end{alignat}
  \end{subequations}
\end{lemma}

\begin{proof}
  \underline{(i) \emph{Proof of \eqref{eq:RW.RV1}.}}
  For all $v \in V_i$, we have
  \[
  \sRed{i}\ReW{V}{i}v \overset{\eqref{eq:def.ReV}}=\sRed{i}(\ReW{W}{i}\Red{i}v,\Pi_{C_i} v)\overset{\eqref{eq:def.sRed}}= \ReW{W}{i}\Red{i}v.
  \]
  \underline{(ii) \emph{Proof of \eqref{eq:RW.RV2}.}}
  For all $v_w \in W_i$, it holds
  \begin{equation*}
    \ReW{V}{i}\Ext{i}v_w
    \overset{\eqref{eq:def.ReV}}= (\ReW{W}{i}\Red{i}\Ext{i}v_w,\Pi_{C_i}\Ext{i}v_w)
    \overset{\ref{P:Identity.ER},\,\eqref{eq:ker.Pi}}= (\ReW{W}{i}v_w,0)
    \overset{\eqref{eq:def.sExt}}=\sExt{i}\ReW{W}{i}v_w.
  \end{equation*}
  \underline{(iii) \emph{Proof of \eqref{eq:EW.EV1}.}}
  For all $\widehat{v} = (\widehat{v}_w,\widehat{v}_c) \in\widehat{V}_i$, we have:
  \begin{equation*}
    \EX{W}{i}\sRed{i}\widehat{v}
    \overset{\eqref{eq:def.sRed}}=
    \EX{W}{i}\widehat{v}_w
    \overset{\ref{P:Identity.ER}}=
    \Red{i}\Ext{i}\EX{W}{i}\widehat{v}_w + \Red{i}\widehat{v}_c
    = \Red{i}(\Ext{i}\EX{W}{i}\widehat{v}_w + \widehat{v}_c)
    \overset{\eqref{eq:def.EX}}=
    \Red{i}\EX{V}{i}\widehat{v},
  \end{equation*}
  where we have additionally used the fact that $\widehat{v}_c \in C_i$ to add $\Red{i} \widehat{v}_c = 0$ in the right-hand side of the second equality and the linearity of $\Red{i} $ in the third equality.
  \medskip\\
  \underline{(iv) \emph{Proof of \eqref{eq:EW.EV2}.}} 
  For all $\widehat{v}_w \in \widehat{W}_i$, we can write
  \begin{equation*}
    \EX{V}{i} \sExt{i} \widehat{v}_w 
    \overset{\eqref{eq:def.sExt}}=
    \EX{V}{i}(\widehat{v}_w,0) 
    \overset{\eqref{eq:def.EX}}=
    \Ext{i}\EX{W}{i} \widehat{v}_w .
  \end{equation*}
  \underline{(v) \emph{Proof of \eqref{eq:part.d}.}}
  For all $\widehat v =(\widehat{v}_w,\widehat{v}_c) \in \widehat{V}_i$, we have:
  \begin{equation*}
    \widehat{\partial}_i\sRed{i} \widehat v
    \overset{\eqref{eq:def.sRed}}= \widehat{\partial}_i \widehat{v}_w
    \overset{\eqref{eq:def.sRed}}= \sRed{i+1}(\widehat{\partial}_i \widehat{v}_w, d_i \widehat{v}_c) 
    \overset{\eqref{eq:def.d}}= \sRed{i+1} \widehat{d}_i\widehat v.\qedhere
  \end{equation*}
\end{proof}

\begin{theorem}[Homological properties for $(V_i,d_i)_i$ and $(\widehat{V}_i,\widehat{d}_i)$]
  Under Assumptions \ref{Isom.W.W} and \ref{Isom.W.V}, the operators $\ReW{V}{i}$ and $\EX{V}{i}$ satisfy the following properties: 
  \begin{subequations}\label{lem:serV.prop}
    \begin{gather}
      (\ReW{V}{i}\EX{V}{i})_{|\Kernel\widehat{d}_i} = \Id_{\Kernel\widehat{d}_i},
      \label{eq:serV.prop.1}\\
      (\EX{V}{i+1}\ReW{V}{i+1}-\Id_{V_{i+1}})(\Kernel d_{i+1}) \subset \Image(d_i),
      \label{eq:serV.prop.2}\\
      \ReW{V}{i+1} d_i=\widehat{d}_i\ReW{V}{i}\text{ and }\EX{V}{i+1} \widehat{d}_i=d_i\EX{V}{i}.
      \label{eq:serV.prop.3}
    \end{gather}
  \end{subequations}
\end{theorem}

\begin{proof}
  \underline{(i) \emph{Proof of \eqref{eq:serV.prop.1}.}}
  Let $\widehat{v}=(\widehat{v}_w,\widehat{v}_c) \in \Kernel \widehat{d}_i$.
  We have
  \[
  \begin{aligned}
    \ReW{V}{i}\EX{V}{i}(\widehat{v}_w,\widehat{v}_c) 
    \overset{\eqref{eq:def.EX}}&=
    \ReW{V}{i}(\Ext{i}\EX{W}{i}\widehat{v}_w+\widehat{v}_c) \\
    \overset{\eqref{eq:def.ReV}}&=
    \big(
    \ReW{W}{i}\Red{i}\Ext{i}\EX{W}{i}\widehat{v}_w,
    \Pi_{C_i} (\Ext{i}\EX{W}{i}\widehat{v}_w+\widehat{v}_c)\big)
    \\
    \overset{\ref{P:Identity.ER},\,\eqref{eq:def.Pi}}&=
    (\ReW{W}{i}\EX{W}{i}\widehat{v}_w,\widehat{v}_c) \\
    \overset{\ref{P:Identity.ERW}}&=
    (\widehat{v}_w,\widehat{v}_c),
  \end{aligned}
  \]
  where we have used the linearity of $\Red{i}$ along with $\Red{i} \widehat{v}_c = 0$ (since $\widehat{v}_c \in C_i$) in the second equality, while the use of \ref{P:Identity.ERW} in the fourth equality is possible since $\widehat{v}_w \in \Kernel \widehat{\partial}_i$, as can be checked writing
  $\widehat{\partial_i} \widehat{v}_w \overset{\eqref{eq:def.sRed}}= \widehat{\partial_i}\sRed{i} \widehat{v} \overset{\eqref{eq:part.d}}= \sRed{i+1}\widehat{d}_i \widehat{v} = 0$, the conclusion being a consequence of $\widehat{v} \in \Kernel \widehat{d}_i$ and the linearity of $\sRed{i+1}$.
  \medskip\\
  \underline{(ii) \emph{Proof of \eqref{eq:serV.prop.2}.}}
  Let
  \[
  \text{
    $v \overset{\eqref{eq:def.v}}= \Ext{i+1} v_w + v_c \in \Kernel d_{i+1}$ with $(v_w, v_c) \in W_{i+1} \times C_{i+1}$.
  }
  \]
  We write
  \begin{equation}
    \begin{aligned}
      \EX{V}{i+1}\ReW{V}{i+1} v - v 
      &=
      \EX{V}{i+1}\ReW{V}{i+1}(\Ext{i+1} v_w+ v_c)-(\Ext{i+1} v_w+ v_c) \\
      \overset{\eqref{eq:def.ReV}}&=
      \EX{V}{i+1}\big(\ReW{W}{i+1} \Red{i+1}(\Ext{i+1} v_w+v_c),\Pi_{C_i}(\Ext{i+1} v_w+ v_c)\big) - (\Ext{i+1} v_w+ v_c) \\
      \overset{\eqref{eq:def.Pi}}&=
      \EX{V}{i+1}(\ReW{W}{i+1} \Red{i+1}\Ext{i+1} v_w,v_c) - (\Ext{i+1} v_w+ v_c) \\
      \overset{\ref{P:Identity.ER}}&=
      \EX{V}{i+1}(\ReW{W}{i+1} v_w, v_c) - (\Ext{i+1} v_w+ v_c) \\
      \overset{\eqref{eq:def.EX}}&=
      \Ext{i+1} \EX{W}{i+1}\ReW{W}{i+1} v_w+ v_c- (\Ext{i+1} v_w+ v_c) \\
      &=
      \Ext{i+1} (\EX{W}{i+1}\ReW{W}{i+1} v_w- v_w),
    \end{aligned}
    \label{eq:serV.prop.2.P1}
  \end{equation}
  where, in the third equality, we have additionally used the fact that $\Red{i+1} v_c = 0$ since $v_c \in C_{i+1}$.
  We next notice that $\Red{i+1}v = \Red{i+1}(\Ext{i+1}v_w+v_c) = \Red{i+1}\Ext{i+1}v_w\overset{\ref{P:Identity.ER}}=v_w$.
  This implies, in turn, $\partial_{i+1}v_w=\partial_{i+1}\Red{i+1}v\overset{\ref{P:E.R.cochain}}=\Red{i+2}d_{i+1}v=\Red{i+2}0=0$ since $v\in\Kernel d_{i+1}$ and $\Red{i+2}$ is linear by definition, giving that $v_w\in\Kernel \partial_i$.
  We can therefore use Assumption \ref{P:complex.ERW} on $\EX{W}{i+1}\ReW{W}{i+1} v_w- v_w$ in \eqref{eq:serV.prop.2.P1} to infer the existence of $q\in W_i$ such that
  \begin{equation*}
    \EX{V}{i+1}\ReW{V}{i+1} v - v 
    =\Ext{i+1} \partial_i q
    \overset{\ref{P:E.R.cochain}}= d_i \Ext{i} q
    \in \Image(d_i).
  \end{equation*}
  \underline{(iii) \emph{Proof of \eqref{eq:serV.prop.3}.}}
  For all $v \in V_i$, we have
  \begin{equation*}
    \begin{aligned}
      \ReW{V}{i+1} d_i v
      \overset{\eqref{eq:def.ReV}}&=
      ( \ReW{W}{i+1} \Red{i+1} d_i v, \Pi_{C_{i+1}}d_i v ) \\
      \overset{\ref{P:E.R.cochain}}&=
      ( \ReW{W}{i+1}\partial_i \Red{i} v ,\Pi_{C_{i+1}}d_i v) \\
      \overset{\ref{P:E.R.W.cochain},~\eqref{eq:d.Pi}}&=
      ( \widehat{\partial}_i \ReW{W}{i} \Red{i}v, d_i \Pi_{C_i} v ) \\
      \overset{\eqref{eq:def.d}}&=
      \widehat{d}_i (\ReW{W}{i} \Red{i} v , \Pi_{C_i} v ) \\
      \overset{\eqref{eq:def.ReV}}&=\widehat{d}_i \ReW{V}{i} v.
    \end{aligned}
  \end{equation*}
  For all $\widehat{v}=(\widehat{v}_w,\widehat{v}_c) \in \widehat{V}_i$, on the other hand, we have:
  \[
  \begin{aligned}
    \EX{V}{i+1} \widehat{d}_i \widehat{v}
    \overset{\eqref{eq:def.d}}&= \EX{V}{i+1} (\widehat{\partial}_i \widehat{v}_w, d_i\widehat{v}_c )
    \\
    \overset{\eqref{eq:def.EX}}&= \Ext{i+1} \EX{W}{i+1} \widehat{\partial}_i \widehat{v}_w+ d_i \widehat{v}_c
    \\
    \overset{\ref{P:E.R.W.cochain},\,\ref{P:E.R.cochain}}&=  d_i \Ext{i} \EX{W}{i} \widehat{v}_w+ d_i \widehat{v}_c
    \\
    \overset{\eqref{eq:def.EX}}&= d_i \EX{V}{i} (\widehat{v}_w, \widehat{v}_c),
  \end{aligned}
  \]
where the conclusion additionally uses the linearity of $d_i$.
\end{proof}

\begin{corollary}[Isomorphism in cohomology]\label{Isom.V.V}
  Under Assumptions \ref{Isom.W.W} and \ref{Isom.W.V}, the cohomologies of all the complexes in diagram \eqref{eq:abstract.diagram} are isomorphic.
\end{corollary}

\begin{proof}
  Theorem \ref{lem:serV.prop} gives all the properties needed to invoke \cite[Proposition~2]{Di-Pietro.Droniou:23*1} and prove that the cohomology of the complex $(\widehat{V}_i,\widehat{d}_i)_i$ is isomorphic to that of $(V_i,d_i)_i$. The latter is, on the other hand, isomorphic to both the cohomologies of $(W_i, \partial_i)$ and $(\widehat{W}_i,\widehat{\partial_i})_i$ (see Remark~\ref{rem:iso.cohom}).
\end{proof}


\section{The discrete de Rham complex and its serendipity version}\label{sec:ddr.sddr}

In this section we recall the Discrete De Rham (DDR) complex of \cite{Di-Pietro.Droniou:23*2} and its serendipity version (SDDR) of \cite{Di-Pietro.Droniou:23*1}.
These complexes will respectively play the role of $(W_i,\partial_i)_i$ and $(\widehat{W}_i,\widehat{\partial}_i)_i$ in \eqref{eq:abstract.diagram} for the applications of the following sections.
We only give a brief overview of the construction for the sake of conciseness and refer to \cite{Di-Pietro.Droniou:23*2,Di-Pietro.Droniou:23*1} for additional details.

\subsection{Local polynomial spaces and $L^2$-orthogonal projectors}\label{sec:ddr.sddr:local.polynomial.spaces}

For a polytope $T_d$ embedded in $\mathbb{R}^n$ with $n \ge d$ and an integer $\ell\ge 0$, we denote by $\Poly{\ell}(T_d)$ the space spanned by the restriction to $T_d$ of $n$-variate polynomials.
Introducing the boldface notation for the space of tangential polynomials $\vPoly{\ell}(T_d) \coloneq \Poly{\ell}(T_d; \Real^d)$ for $d \in \{ 2, 3\}$, the following direct decompositions hold (see, e.g., \cite{Arnold:18}):
\[
\begin{gathered}
  \vPoly{\ell}(T_2) = \Goly{\ell}(T_2) \oplus \cGoly{\ell}(T_2)
  \\
  \text{%
    with $\Goly{\ell}(T_2)\coloneq\GRAD_{T_2}\Poly{\ell+1}(T_2)$
    and $\cGoly{\ell}(T_2)\coloneq(\bvec{x}-\bvec{x}_{T_2})^\perp\Poly{\ell-1}(T_2)$,
  }
\end{gathered}
\]
where $\GRAD_{T_2}$ denotes the tangential gradient when $T_2$ is embedded in $\Real^3$ and $\bvec{v}^\perp$ is obtained rotating $\bvec{v}$ by $\frac{\pi}{2}$,
\[
\begin{gathered}
  \vPoly{\ell}(T_3) = \Goly{\ell}(T_3) \oplus \cGoly{\ell}(T_3)
  \\
  \text{%
    with $\Goly{\ell}(T_3)\coloneq\GRAD\Poly{\ell+1}(T_3)$
    and $\cGoly{\ell}(T_3)\coloneq(\bvec{x}-\bvec{x}_{T_3})\times \vPoly{\ell-1}(T_3)$,
  }
\end{gathered}
\]
and, for $d \in \{ 2, 3 \}$,
\[
\begin{gathered}
  \vPoly{\ell}(T_d) = \Roly{\ell}(T_d) \oplus \cRoly{\ell}(T_d)
  \\
  \text{%
    with $\Roly{\ell}(T_d)\coloneq\VROT_{T_d}\Poly{\ell+1}(T_d)$
    and $\cRoly{\ell}(T_d)\coloneq(\bvec{x}-\bvec{x}_{T_d})\Poly{\ell-1}(T_d)$,
  }
\end{gathered}
\]
where $\VROT_{T_2} \coloneq \GRAD_{T_2}^\perp$ and $\VROT_{T_3} \coloneq \CURL$.

We extend the above notations to negative exponents $\ell$ by setting all the spaces appearing in the decompositions equal to the trivial vector space.
Given a polynomial (sub)space $\mathcal{X}^\ell(T_d)$, the corresponding $L^2$-orthogonal projector is denoted by $\pi_{\mathcal{X},T_d}^\ell$.
Boldface font will be used when the elements of $\mathcal{X}^\ell(T_d)$ are vector-valued, and, for $\cvec{X} \in \{ \cvec{R}, \cvec{G} \}$, $\Xcproj{\ell}{T_d}$ denotes the $L^2$-orthogonal projector on $\cvec{X}^{{\rm c},\ell}(T_d)$.

\subsection{The two-dimensional discrete de Rham complex}

\subsubsection{Spaces}\label{sec:ddr:2d:spaces}

Given a two-dimensional polygonal mesh $\mathcal{M}_h$, we denote by $\mathcal{M}_{0,h}$, $\mathcal{M}_{1,h}$ and $\mathcal{M}_{2,h}$, respectively, the set of vertices $T_0$, edges $T_1$, and elements $T_2$ of the mesh.
Let $k \geq 0$ be a given polynomial degree and, for all $T_2 \in \mathcal{M}_{2,h}$, $n_{T_2}$ and $s_{T_2}$ two integers $\ge -1$ that we collect in the vectors $\boldsymbol{n}= ( n_{T_2} )_{T_2\in \mathcal{M}_{2,h}}$ and $\boldsymbol{s}=( s_{T_2} )_{T_2\in \mathcal{M}_{2,h}}$. The boldface notation is dropped when the values in $\boldsymbol{n}$ and $\boldsymbol{s}$ are all equal.\\

We define the following discrete counterparts of $H^1(\Omega)$, $\Hrot{\Omega}$, and $L^2(\Omega)$:
\[
\Xgrad{\bvec{n},k}{h}\coloneq\Big\{
\begin{aligned}[t]
  \underline{q}_{h}
  &=\big(
  (q_{T_2})_{T_2\in\mathcal{M}_{2,h}},(q_{T_1})_{T_1\in\mathcal{M}_{1,h}}, (q_{T_0})_{T_0\in\mathcal{M}_{0,h}}
  \big)\st
  \\
  &\qquad\text{$q_{T_2}\in\Poly{n_{T_2}}(T_2)$ for all $T_2\in\mathcal{M}_{2,h}$,}
  \\
  &\qquad\text{$q_{T_1}\in\Poly{k-1}(T_1)$ for all $T_1\in\mathcal{M}_{1,h}$,}
  \\
  &\qquad\text{$q_{T_0}\in\Real$ for all $T_0\in\mathcal{M}_{0,h}$}
  \Big\},
\end{aligned}
\]
\[
\Xcurl{\bvec{s},k}{h}\coloneq\Big\{
\begin{aligned}[t]
  \uvec{v}_{w,h}
  &=\big(
  (\bvec{v}_{\cvec{R},T_2},\bvec{v}_{\cvec{R},T_2}^\compl)_{T_2\in\mathcal{M}_{2,h}}, (v_{T_1})_{T_1\in\mathcal{M}_{1,h}}
  \big)\st
  \\
  &\qquad\text{$\bvec{v}_{\cvec{R},T_2}\in\Roly{k-1}(T_2)$ and $\bvec{v}_{\cvec{R},T_2}^\compl\in\cRoly{s_{T_2}}(T_2)$ for all $T_2\in\mathcal{M}_{2,h}$,}
  \\
  &\qquad\text{and $v_{T_1}\in\Poly{k}(T_1)$ for all $T_1\in\mathcal{M}_{1,h}$}\Big\},
\end{aligned}
\]
\[
\XL{k}{h}\coloneq\Poly{k}(\mathcal{M}_{2,h}),
\]
where $\Poly{k}(\mathcal{M}_{2,h})$ denotes the space of broken polynomials on $\mathcal{M}_{2,h}$ of total degree $\le k$.
The restriction of $\Xgrad{\bvec{n},k}{h}$ to an element $T_d$, $d \in \{ 1, 2\}$, is obtained collecting the components on $T_d$ and its boundary and is denoted by $\Xgrad{n,k}{T_d}$.
Similar conventions are used for the restriction of the spaces that will appear in the rest of the paper as well as their elements.

\subsubsection{Discrete vector calculus operators}

For any edge $T_1\in\mathcal{M}_{1,T_2}$ and any $\underline{q}_{T_1} \in \Xgrad{k-1,k}{T_1}$, the edge gradient $\GE \underline{q}_{T_1}$ is defined as the derivative along $T_1$ of the function $\trE\underline{q}_{T_1} \in \Poly{k+1}(T_1)$ such that $\trE\underline{q}_{T_1}(\bvec{x}_{T_0}) = q_{T_0}$ for any vertex $T_0$ of $T_1$ of coordinates $\bvec{x}_{T_0}$ and $\lproj{k-1}{T_1} \trE\underline{q}_{T_1} = q_{T_1}$.
We next define the gradient $\GF:\Xgrad{k-1,k}{T_2}\to\vPoly{k}(T_2)$ and the scalar two-dimensional potential $\trF:\Xgrad{k-1,k}{T_2}\to\Poly{k+1}(T_2)$ on $T_2$ such that, for all $\underline{q}_{T_2}\in\Xgrad{k-1,k}{T_2}$,
\[
\int_{T_2}\GF\underline{q}_{T_2}\cdot\bvec{v}
= -\int_{T_2} q_{T_2}\DIV_F\bvec{v}
+ \sum_{T_1\in\mathcal{M}_{1,T_2}}\omega_{T_2T_1}\int_{T_1} \trE\underline{q}_{T_1}~(\bvec{v}\cdot\normal_{T_2T_1})
\quad\forall\bvec{v}\in\vPoly{k}(T_2),
\]
\begin{equation} \label{eq:trF}
  \begin{multlined}
    \int_{T_2}\trF\underline{q}_{T_2}\DIV_{T_2}\bvec{v}
    = -\int_{T_2}\GF\underline{q}_{T_2}\cdot\bvec{v}
    + \sum_{T_1\in\mathcal{M}_{1,T_2}}\omega_{T_2T_1}\int_{T_1} \trE\underline{q}_{T_1}~(\bvec{v}\cdot\normal_{T_2T_1})
    \\
    \forall\bvec{v}\in\cRoly{k+2}(T_2),
  \end{multlined}
\end{equation}
  where $\normal_{T_2T_1}$ is a unit normal vector to $T_1$ lying in the plane of $T_2$ and $\omega_{T_2 T_1}$ the orientation of $T_1$ relative to $T_2$ such that $\omega_{T_2T_1} \normal_{T_2T_1}$ points out of $T_2$.

The two-dimensional scalar rotor $\CF:\Xcurl{k,k}{T_2}\to\Poly{k}(T_2)$ and the corresponding vector potential $\trFt:\Xcurl{k,k}{T_2}\to\vPoly{k}(T_2)$ (which can be interpreted as a tangential component when $T_2$ is the face of a polyhedron) are such that, for all $\uvec{v}_{T_2}\in\Xcurl{k,k}{T_2}$,
\[
\int_{T_2}\CF\uvec{v}_{T_2}~r
= \int_{T_2}\bvec{v}_{\cvec{R},{T_2}}\cdot\VROT_{T_2} r
- \sum_{T_1\in\mathcal{M}_{1,T_2}}\omega_{T_2T_1}\int_{T_1} v_{T_1}~r
\qquad\forall r\in\Poly{k}(T_2),
\]
\begin{equation}\label{eq:trFt}
  \begin{multlined}
    \int_{T_2}\trFt\uvec{v}_{T_2}\cdot(\VROT_{T_2} r + \bvec{w})
    = \int_{T_2}\CF\uvec{v}_{T_2}~r
    + \sum_{T_1\in\mathcal{M}_{1,T_2}}\omega_{T_2T_1}\int_{T_1} v_{T_1}~r
    + \int_{T_2}\bvec{v}_{\cvec{R},{T_2}}^\compl\cdot\bvec{w}
    \\
    \forall (r,\bvec{w})\in\Poly{k+1}(T_2)\times\cRoly{k}(T_2).
  \end{multlined}
\end{equation}
We will also need the two-dimensional vector rotor $\bvec{C}^k_{T_2}:\Xcurl{k,k}{T_2}\to\vPoly{k}(T_2)$ such that
\begin{equation} \label{eq:vCF}
  \int_{T_2} \bvec{C}^k_{T_2} \uvec{v}_{T_2} \cdot \bvec{w} = \int_{T_2} v_{T_2} \ROT \bvec{w} + \sum_{T_1\in\mathcal{M}_{1,T_2}} \omega_{T_2T_1} \int_{T_1} (\bvec{v}_{T_1} \cdot \normal_{T_2T_1}) (\bvec{w} \cdot \tangent_{T_1}) 
  \qquad\forall \bvec{w} \in \vPoly{k}(T_2).
\end{equation}

\subsubsection{DDR complex}\label{sec:ddr.2d:complex}

The two-dimensional DDR complex of degree $k$ reads
\begin{equation} \label{eq:glob.2D-ddr.complex}
  \begin{tikzpicture}[xscale=2,baseline={(Xgrad.base)}]
    \node at (-1,0) {DDR2d:};
    
    \node (Xgrad) at (0,0) {$\Xgrad{k-1,k}{h}$};
    \node (Xrot) at (1.5,0) {$\Xrot{k,k}{h}$};
    \node (WL2) at (3,0) {$\XL{k}{h}$,};

    \draw[->,>=latex] (Xgrad) -- (Xrot) node[midway, above, font=\footnotesize]{$\Gddr{h}$};
    \draw[->,>=latex] (Xrot) -- (WL2) node[midway, above, font=\footnotesize]{$\rotddr{h}$};
  \end{tikzpicture}
\end{equation}
where the discrete global gradient $\Gddr{h}$ and curl $\rotddr{h}$ are such that, for all $(\underline{q}_{h}, \uvec{v}_{h}) \in \Xgrad{k-1,k}{h} \times \Xcurl{k,k}{h}$,
\[
\begin{gathered}
  \Gddr{h} \underline{q}_{h}
  \coloneq
  \big(
  (\Rproj{k-1}{T_2}\GF\underline{q}_{T_2},\cRproj{k}{T_2}\GF\underline{q}_{T_2})_{T_2\in\mathcal{M}_{2,h}},
  ( \GE q_{T_1} )_{T_1\in\mathcal{M}_{1,h}}
  \big),
  \\
    \text{$( \rotddr{h} \uvec{v}_{h} )_{| T_2} \coloneq \CF\uvec{v}_{T_2}$ for all $T_2\in\mathcal{M}_{2,h}$}.
\end{gathered}
\]

\subsection{The three-dimensional discrete de Rham complex}

\subsubsection{Spaces}\label{sec:ddr:3d:spaces}

Let us now consider a three-dimensional mesh $\mathcal{M}_h$, with $\mathcal{M}_{0,h}$, $\mathcal{M}_{1,h}$, $\mathcal{M}_{2,h}$, and $\mathcal{M}_{3,h}$ denoting, respectively, the set of vertices $T_0$, edges $T_1$, faces $T_2$, and elements $T_3$.
Given four vectors of integers $\ge -1$ $\bvec{m}\coloneq(m_{T_3})_{T_3\in\mathcal{M}_{3,h}}$, $\bvec{n}\coloneq(n_{T_2})_{T_2\in\mathcal{M}_{2,h}}$, $\bvec{p}\coloneq(p_{T_3})_{T_3\in\mathcal{M}_{3,h}}$, and $\bvec{s}\coloneq(s_{T_2})_{T_2\in\mathcal{M}_{3,h}}$, we define the following discrete counterparts of $H^1(\Omega)$, $\Hcurl{\Omega}$, $\Hdiv{\Omega}$, and $L^2(\Omega)$:
\[
\Xgrad{\bvec{m},\bvec{n},k}{h}\coloneq\Big\{
\begin{aligned}[t]
  \underline{q}_{w,h}
  &=\big(
  (q_{T_3})_{T_3\in\mathcal{M}_{3,h}},(q_{T_2})_{T_2\in\mathcal{M}_{2,h}},(q_{T_1})_{T_1\in\mathcal{M}_{1,h}}, (q_{T_0})_{T_0\in\mathcal{M}_{0,h}}
  \big)\st
  \\
  &\qquad\text{$q_{T_3}\in\Poly{m_{T_3}}(T_3)$for all $T_3 \in\mathcal{M}_{3,h}$,}
  \\
  &\qquad\text{$q_{T_2}\in\Poly{n_{T_2}}(T_2)$ for all $T_2 \in\mathcal{M}_{2,h}$,}
  \\
  &\qquad\text{$q_{T_1}\in\Poly{k-1}(T_1)$ for all $T_1\in\mathcal{M}_{1,h}$,}
  \\
  &\qquad\text{and $q_{T_0}\in\Real$ for all $T_0\in\mathcal{M}_{0,h}$}
  \Big\},
\end{aligned}
\]
\[
\Xcurl{\bvec{p},\bvec{s},k}{h}\coloneq\Big\{
\begin{aligned}[t]
  \uvec{v}_{w,h}
  &=\big(
  (\bvec{v}_{\cvec{R},T_3},\bvec{v}_{\cvec{R},T_3}^\compl)_{T_3\in\mathcal{M}_{3,h}}, (\bvec{v}_{\cvec{R},T_2},\bvec{v}_{\cvec{R},T_2}^\compl)_{T_2\in\mathcal{M}_{2,h}}, (v_{T_1})_{T_1\in\mathcal{M}_{1,h}}
  \big)\st
  \\
  &\qquad\text{$\bvec{v}_{\cvec{R},T_3}\in\Roly{k-1}(T_3)$ and $\bvec{v}_{\cvec{R},T_3}^\compl\in\cRoly{p_{T_3}}(T_3)$ for all $T_3\in\mathcal{M}_{3,h}$,}
  \\
  &\qquad\text{$\bvec{v}_{\cvec{R},T_2}\in\Roly{k-1}(T_2)$ and $\bvec{v}_{\cvec{R},T_2}^\compl\in\cRoly{s_{T_2}}(T_2)$ for all $T_2\in\mathcal{M}_{2,h}$,}
  \\
  &\qquad\text{and $v_{T_1}\in\Poly{k}(T_1)$ for all $T_1\in\mathcal{M}_{1,h}$}\Big\},
\end{aligned}
\]
\[
\Xdiv{h}\coloneq\Big\{
\begin{aligned}[t]
  \uvec{w}_{w,h}
  &=\big((\bvec{w}_{\cvec{G},T_3},\bvec{w}_{\cvec{G},T_3}^\compl)_{T_3\in\mathcal{M}_{3,h}}, (w_{T_2})_{T_2\in\mathcal{M}_{2,h}}\big)\st
  \\
  &\qquad\text{$\bvec{w}_{\cvec{G},T_3}\in\Goly{k-1}(T_3)$ and $\bvec{w}_{\cvec{G},T_3}^\compl\in\cGoly{k}(T_3)$ for all $T_3\in\mathcal{M}_{3,h}$,}
  \\
  &\qquad\text{and $w_{T_2}\in\Poly{k}(T_2)$ for all $T_2\in\mathcal{M}_{2,T_3}$}
  \Big\},
\end{aligned}
\]
and
\[
\XL{k}{h} \coloneq \Poly{k}(\mathcal{M}_{3,h}).
\]
  When the values in $\boldsymbol{m}$, $\boldsymbol{n}$, $\boldsymbol{p}$ and $\boldsymbol{s}$ are all equal, where we drop the boldface notation.
  With a little abuse in notation, for the discrete gradient operator defined by \eqref{eq:ddr.3d:operators} below as well as for the tail space $\XL{k}{h}$, we use the same symbols as for the DDR2d sequence: all ambiguity will be removed by the context.

\subsubsection{Discrete vector calculus operators}

The element gradient $\GT:\Xgrad{k-1,k-1,k}{T_3}\to\vPoly{k}(T_3)$, the element curl $\CT:\Xcurl{k,k,k}{T_3}\to\vPoly{k}(T_3)$, and the element divergence $\DT:\Xdiv{T_3}\to\Poly{k}(T_3)$ are respectively defined such that, for all $\underline{q}_{T_3}\in\Xgrad{k-1,k-1,k}{T_3}$, all $\uvec{v}_{T_3}\in\Xcurl{k,k,k}{T_3}$, and all $\uvec{w}_{T_3}\in\Xdiv{T_3}$,
\begin{gather}\label{eq:GT}
  \int_{T_3}\GT\underline{q}_{T_3}\cdot\bvec{v}
  = -\int_{T_3} q_{T_3}\DIV\bvec{v}
  + \sum_{T_2\in\mathcal{M}_{2,T_3}}\omega_{T_3T_2}\int_{T_2}\trF\underline{q}_{T_2}~(\bvec{v}\cdot\normal_{T_2})
  \quad\forall\bvec{v}\in\vPoly{k}(T_3),
  \\\label{eq:CT}
  \begin{multlined}
    \int_{T_3}\CT\uvec{v}_{T_3}\cdot\bvec{z}
    = \int_{T_3}\bvec{v}_{\cvec{R},T_3}\cdot\CURL\bvec{z}
    + \sum_{T_2\in\mathcal{M}_{2,T_3}}\omega_{T_3T_2}\int_{T_2}\trFt\uvec{v}_{T_2}\cdot(\bvec{z}\times\normal_{T_2})
    \\
    \forall\bvec{z}\in\vPoly{k}(T_3),
  \end{multlined}\\ \label{eq:DT}
  \int_{T_3}\DT\uvec{w}_{T_3}~q
  = -\int_{T_3}\bvec{w}_{\cvec{G},T_3}\cdot\GRAD q
  + \sum_{T_2\in\mathcal{M}_{2,T_3}}\omega_{T_3T_2}\int_{T_2} w_{T_2}~q
  \qquad\forall q\in\Poly{k}(T_3),
\end{gather}
where $\normal_{T_2}$ is a unit normal vector to $T_2$ and $\omega_{T_3T_2}$ is the orientation of $T_2$ relative to $T_3$ such that $\omega_{T_3 T_2} \normal_{T_2}$ points out of $T_3$.

\subsubsection{DDR complex}

The global three-dimensional DDR complex of degree $k$ is
\begin{equation} \label{eq:glob.3d-ddr.complex}
  \begin{tikzpicture}[xscale=2,baseline={(Xgrad.base)}]
    \node at (-1,0) {DDR3d:};
    
    \node (Xgrad) at (0,0) {$\Xgrad{k-1,k-1,k}{h}$};
    \node (Xcurl) at (1.5,0) {$\Xcurl{k,k,k}{h}$};
    \node (Xdiv) at (3,0) {$\Xdiv{h}$};
    \node (WL2) at (4.5,0) {$\XL{k}{h}$,};

    \draw[->,>=latex] (Xgrad) -- (Xcurl) node[midway, above, font=\footnotesize]{$\Gddr{h}$};
    \draw[->,>=latex] (Xcurl) -- (Xdiv) node[midway, above, font=\footnotesize]{$\Cddr{h}$};
    \draw[->,>=latex] (Xdiv) -- (WL2) node[midway, above, font=\footnotesize]{$\Dh{h}$};
  \end{tikzpicture}
\end{equation}
where the operators $\Gddr{h}$, $\Cddr{h}$ and $\Dh{h}$ are obtained projecting the element and face operators onto the component spaces:
For all $(\underline{q}_{h},\uvec{v}_{h},\uvec{w}_{h})\in\Xgrad{k-1,k-1,k}{h}\times\Xcurl{k,k,k}{h}\times\Xdiv{h}$,
\begin{equation}\label{eq:ddr.3d:operators}
  \begin{gathered}
    \Gddr{h} \underline{q}_{h}
    \coloneq
    \begin{aligned}[t]
      \big(
      &(\Rproj{k-1}{T_3}\GT\underline{q}_{T_3},\cRproj{k}{T_3}\GT\underline{q}_{T_3})_{T_3\in\mathcal{M}_{3,h}},
      \\
      &( \Rproj{k-1}{T_2}\GF\underline{q}_{T_2},\cRproj{k}{T_2}\GF\underline{q}_{T_2} )_{T_2\in\mathcal{M}_{2,h}},
      \\
      &( \GE q_{T_1} )_{T_1\in\mathcal{M}_{1,h}}
      \big),
    \end{aligned} \\
    \Cddr{h}\uvec{v}_{h}
    \coloneq\big(
    (\Gproj{k-1}{T_3}\CT\uvec{v}_{T_3},\Gcproj{k}{T_3}\CT\uvec{v}_{T_3})_{T_3\in\mathcal{M}_{3,h}},
    ( \CF\uvec{v}_{T_2} )_{T_2\in\mathcal{M}_{2,h}}
    \big),
    \\
    \text{%
      $( \Dh{h}\uvec{w}_{h} )_{| T_3} \coloneq \DT\uvec{w}_{T_3}$ for all $T_3\in\mathcal{M}_{3,h}$.
    }
  \end{gathered}
\end{equation}

\subsection{Serendipity spaces}

We now introduce the two- and three-dimensional Serendipity Discrete de Rham (SDDR) complexes that will play the role of $(\widehat{W}_i,\widehat{\partial}_i)_i$ in the applications considered in Sections~\ref{sec:rotrot} and \ref{sec:stokes} below.

For each $T_d \in \mathcal{M}_{d,h}$, $d \in \{2, 3\}$, we select $\eta_{T_d}\ge 2$ faces/edges that are not pairwise aligned and such that ${T_d}$ lies entirely on one side of the plane/line spanned by each of those faces/edges and the regularity assumption detailed in \cite[Assumption~12]{Di-Pietro.Droniou:23*1} are satisfied.
We then set
\[
\ell_{T_d} \coloneq k + 1 - \eta_{T_d}.
\]
These integers are collected in the vector $\boldsymbol{\ell}_d \coloneq ( \ell_{T_d} )_{T_d \in \mathcal{M}_{d,h}}$.
The serendipity version of the spaces in \eqref{eq:glob.2D-ddr.complex} and \eqref{eq:glob.3d-ddr.complex} are, respectively,
\begin{equation}\label{eq:SXgrad.SXcurl}
  \begin{alignedat}{3}
  \SXgrad{k}{h} &\coloneq \Xgrad{\bvec{\ell}_2,k}{h},&\qquad
  \SXrot{h} &\coloneq \Xrot{\bvec{\ell}_2+1,k}{h},
  \\
  \SXgrad{k}{h} &\coloneq \Xgrad{\bvec{\ell}_3,\bvec{\ell}_2,k}{h}, &\qquad
  \SXcurl{h} &\coloneq \Xcurl{\bvec{\ell}_3 +1,\bvec{\ell}_2 +1,k}{h}.
  \end{alignedat}
\end{equation}
In these spaces, the degree of certain polynomial components inside faces and elements for which $\eta_{T_d} > 2$ is lower than in the non-serendipity spaces defined in Sections~\ref{sec:ddr:2d:spaces} and~\ref{sec:ddr:3d:spaces}, the more so the larger $\eta_{T_d}$.

\subsection{Extension and reduction maps between the two-dimensional DDR and SDDR complexes}

Following \cite[Section~5.3]{Di-Pietro.Droniou:23*1}, for a polygon $T_2$ it is possible to define serendipity gradient and rotor operators $\SerGrad{T_2}:\SXgrad{k}{T_2}\to\vPoly{k}(T_2)$ and $\SerRot{T_2}:\SXrot{T_2}\to\vPoly{k}(T_2)$ that satisfy the following properties:
\[
\begin{gathered}
  \SerGrad{T_2}\sIGrad{T_2} q = \GRAD_{T_2}q 
  \qquad \forall q \in \Poly{k+1}(T_2),
  \\
  \SerRot{T_2}\sIrot{T_2} \bvec{v}=\bvec{v}
  \qquad \forall \bvec{v} \in \vPoly{k}(T_2),
\end{gathered}
\]
where $\sIGrad{T_2}$ and $\sIrot{T_2}$ are the standard DDR interpolators on $\SXgrad{k}{T_2}$ and $\SXrot{T_2}$, obtained collecting $L^2$-orthogonal projections on the component spaces.
The role of the serendipity operators is to reconstruct polynomials fields inside $T_2$ from the polynomial components of the serendipity spaces.

In order to define two-dimensional extension maps, we need an operator $\EPoly{T_2}:\SXgrad{k}{T_2}\to\Poly{k-1}(T_2)$ that satisfies a formal integration by parts with the serendipity gradient:
For all $\bvec{w}\in\cRoly{k}(T_2)$,
\[
\int_F\EPoly{T_2}\underline{\widehat{q}}_{T_2}\DIV_{T_2}\bvec{w}
= - \int_{T_2}\SerGrad{T_2}\underline{\widehat{q}}_{T_2}\cdot\bvec{w}
+ \sum_{T_i\in\mathcal{M}_{1,T_2}}\omega_{T_2 T_1}{}\int_{T_1} \widehat{q}_{T_1}~(\bvec{w}\cdot\normal_{T_2T_1}).
\]
The extension operators $\Egrad{h}:\SXgrad{k}{h}\to\Xgrad{k-1,k}{h}$
and $\Erot{h}:\SXrot{h}\to\Xrot{k,k}{h}$ are defined by
\begin{gather} \label{eq:Egrad.T2}
  \Egrad{h}\underline{\widehat{q}}_{h}
  \coloneq \big(
  (\EPoly{T_2}\underline{\widehat{q}}_{T_2})_{T_2 \in \mathcal{M}_{2,h}}, (\widehat{q}_{T_1})_{T_1 \in \mathcal{M}_{1,h}}, (\widehat{q}_{T_0})_{T_0 \in \mathcal{M}_{0,h}}
  \big)
  \qquad\forall\underline{\widehat{q}}_{h}\in\SXgrad{k}{h},
  \\  \label{eq:Erot.T2}
  \Erot{h}\suvec{v}_{h}
  \coloneq \big(
  (\svec{v}_{\cvec{R},T_2}, \cRproj{k}{T_2}\SerRot{T_2} \suvec{v}_{T_2})_{T_2\in\mathcal{M}_{2,h}}, (\widehat{v}_{T_1})_{T_1\in\mathcal{M}_{1,h}}
  \big)
  \qquad\forall\suvec{v}_{h}\in\SXrot{h},
\end{gather}
while the reduction operators $\Rgrad{h}:\Xgrad{k-1,k}{h}\to\SXgrad{k}{h}$ and $\Rrot{h}:\Xrot{k,k}{h} \to\SXrot{h}$ are such that
\begin{gather}\label{eq:Rgrad.T2}
  \Rgrad{h} \underline{q}_{h}
  \coloneq \big(
  (\lproj{\ell_{T_2}}{T_2}q_{T_2})_{T_2 \in \mathcal{M}_{2,h}}, (q_{T_1})_{T_1 \in \mathcal{M}_{1,h}}, (q_{T_0})_{T_0 \in \mathcal{M}_{0,h}}
  \big)
  \qquad\forall\underline{q}_{h}\in\Xgrad{k-1,k}{h},
  \\ \label{eq:Rrot.T2}
  \Rrot{h}\uvec{v}_{h}
  \coloneq \big(
  (\bvec{v}_{\cvec{R},T_2}, \cRproj{\ell_{T_2}+1}{T_2}\bvec{v}_{\cvec{R},T_2}^\compl)_{T_2 \in \mathcal{M}_{2,h}}, (v_{T_1})_{T_1\in\mathcal{M}_{1,T_2}}
  \big)
  \qquad\forall\uvec{v}_{h}\in\Xrot{k,k}{h}.
\end{gather}

The complexes $(W_i,\partial_i)_i $ and $(\widehat{W}_i,\widehat{\partial}_i)_i$ along with the corresponding extension and reduction maps that will be used in the application of Section~\ref{sec:rotrot} are summarized in the following diagram:
\begin{equation}\label{eq:ddr2d.sddr2d}
\begin{tikzpicture}[xscale=2.5, yscale=1.25, baseline={(WL2.base)}]
  \node at (-1,1) {DDR2d:};
  \node at (-1,-1) {SDDR2d:};

  \node (Xgrad) at (0,1) {$\Xgrad{k-1,k}{h}$};
  \node (Xrot) at (1.5,1) {$\Xrot{k,k}{h}$};
  \node (WL2) at (3,0) {$\XL{k}{h}$};

  \draw[->,>=latex] (Xgrad) -- (Xrot) node[midway, above, font=\footnotesize]{$\Gddr{h}$};
  \draw[->,>=latex] (Xrot) -- (WL2) node[midway, above, font=\footnotesize]{$\rotddr{h}$};

  \node (SXgrad) at (0,-1) {$\SXgrad{k}{h}$};
  \node (SXrot) at (1.5,-1) {$\SXrot{h}$};
  
  \draw[->,>=latex] (SXgrad) -- (SXrot) node[midway, above, font=\footnotesize]{$\sGddr{h}$};
  \draw[->,>=latex] (SXrot) -- (WL2) node[midway, below, font=\footnotesize]{$\sRddr{h}$};

  \draw [->,>=latex] (SXgrad) to [bend right=10] node[midway, right, font=\footnotesize] {$\Egrad{h}$} (Xgrad) ;
  \draw [->,>=latex,dashed] (Xgrad) to [bend right=10] node[midway, left, font=\footnotesize] {$\Rgrad{h}$} (SXgrad) ;

  \draw [->,>=latex] (SXrot) to [bend right=10] node[midway, right, font=\footnotesize] {$\Erot{h}$} (Xrot) ;
  \draw [->,>=latex,dashed] (Xrot) to [bend right=10] node[midway, left, font=\footnotesize] {$\Rrot{h}$} (SXrot) ;
\end{tikzpicture}
\end{equation}
where $\sGddr{h}$ and $\sRddr{h}$ are given by \eqref{eq:comm.diag}.

\subsection{Extension and reduction maps between the three-dimensional DDR and SDDR complexes}

Now, taking a polyhedron $T_3$ and following again \cite[Section~5.4]{Di-Pietro.Droniou:23*1}, it is possible to define serendipity gradient and curl operators $\SerGrad{T_3}:\SXgrad{k}{T_3}\to\Poly{k}(T_3)$ and $\SerCurl{T_3}:\SXcurl{T_3}\to\Poly{k}(T_3)$ that satisfy the following properties:
\[
\begin{gathered}
  \SerGrad{T_3}\sIGrad{T_3} q=\GRAD_{T_3}q 
  \qquad \forall q \in \Poly{k+1}(T_3),
  \\
  \SerRot{T_3}\sIcurl{T_3} \bvec{v}=\bvec{v}
  \qquad \forall \bvec{v} \in \vPoly{k}(T_3),
\end{gathered}
\]
  where $\sIGrad{T_3}$ and $\sIcurl{T_3}$ are the standard DDR interpolators on $\SXgrad{k}{h}$ and $\SXcurl{h}$ obtained collecting $L^2$-orthogonal projection on the component spaces.
We also define 
$\EPoly{T}:\SXgrad{k}{T_3}\to\Poly{k-1}(T_3)$ such that, for all $\bvec{w}\in\cRoly{k}(T_3)$,
\[ 
\int_{T_3}\EPoly{T}\underline{\widehat{q}}_{T_3}\DIV\bvec{w}
= - \int_{T_3}\SerGrad{T_3}\underline{\widehat{q}}_T\cdot\bvec{w}
+ \sum_{T_2\in\mathcal{M}_{{T_2}\in {T_3}}}\omega_{T_3T_2}\int_{T_2} \widehat{q}_{T_2}~(\bvec{w}\cdot\normal_{T_2}),
\] 
$\RPoly{T_3}:\Xgrad{k-1,k-1,k}{T_3}\to\Poly{\ell_{T_3}}(T_3)$, such that, for all $\bvec{w}\in\cRoly{\ell_{T_3}+1}(T_3)$,
\begin{multline*}
  \int_{T_3}\RPoly{T_3}\underline{q}_{T_3}\DIV\bvec{w}
  =-\int_{T_3}\GT\underline{q}_{T_3}\cdot\bvec{w}
  \\
  + \sum_{T_2\in\mathcal{M}_{2,T_3}}\omega_{T_3 T_2}\int_{T_2} \trF\Egrad{T_2}\Rgrad{T_2}\underline{q}_{T_2}~(\bvec{w}\cdot\normal_{T_2}),
\end{multline*}
and $\RRoly{T_3}:\SXcurl{T_3}\to\Roly{k-1}(T_3)$ such that, for all $\bvec{w}\in\cGoly{k}(T_3)$,
\begin{multline*}
  \int_{T_3} \RRoly{T_3}\uvec{v}_{T_3}\cdot\CURL\bvec{w}
  = \int_{T_3}\CT\uvec{v}_{T_3}\cdot\bvec{w}
  \\
  - \sum_{T_2\in\mathcal{M}_{2,T_3}}\omega_{T_3T_2}\int_{T_2}\trFt\Ecurl{T_2}\Rcurl{T_2}\uvec{v}_{T_2}\cdot (\bvec{w}\times\normal_{T_2}).
\end{multline*}
where $\trF$, $\trFt$, $\GT$, and $\CT$, are respectively defined by
\eqref{eq:trF}, \eqref{eq:trFt}, \eqref{eq:GT}, and \eqref{eq:CT}.

The extension operators $\Egrad{h}:\SXgrad{k}{h}\to\Xgrad{k-1,k-1,k}{h}$
and $\Ecurl{h}:\SXcurl{h}\to\Xcurl{k,k,k}{h}$ are such that, for all $\underline{\widehat{q}}_{h}\in\SXgrad{k}{h}$ and all $\suvec{v}_{h}\in \SXcurl{h}$,
\[
\begin{gathered}
  \Egrad{h}\underline{\widehat{q}}_{h}
  \coloneq \big(
  (\EPoly{T_3}\underline{\widehat{q}}_{T_3})_{T_3\in\mathcal{M}_{3,h}},
  (\EPoly{T_2}\underline{\widehat{q}}_{T_2})_{T_2\in\mathcal{M}_{2,h}},
  (\widehat{q}_{T_1})_{T_1 \in \mathcal{M}_{1,h}}, (\widehat{q}_{T_0})_{T_0 \in \mathcal{M}_{0,h}}
  \big),
  \\
  \Ecurl{h}\suvec{v}_{h}
  \coloneq \big(
  (\svec{v}_{\cvec{R},T_3}, \cRproj{k}{T_3}\SerCurl{T_3}\suvec{v}_{T_3})_{T_3\in\mathcal{M}_{3,h}},
  (\svec{v}_{\cvec{R},T_2}, \cRproj{k}{T_2}\SerCurl{T_2}\suvec{v}_{T_2})_{T_2\in\mathcal{M}_{2,h}}, (\widehat{v}_{T_1})_{T_1\in\mathcal{M}_{1,h}}
  \big),
\end{gathered}
\]
while the reduction operators are $\Rgrad{h}:\Xgrad{k-1,k-1,k}{h}\to\SXgrad{k}{h}$
and $\Rcurl{h}:\Xcurl{k,k,k}{h}\to\SXcurl{h}$ such that, for all $\underline{q}_{h}\in\Xgrad{k-1,k-1,k}{h}$ and all $\uvec{v}_{h}\in\Xcurl{k,k,k}{h}$, 
\[
\begin{gathered}
  \Rgrad{h}\underline{q}_{h}
  \coloneq \big(
  (\RPoly{T_3}\underline{q}_{T_3})_{T_3\in\mathcal{M}_{3,h}},
  (\lproj{\ell_{T_2}}{T_2}q_{T_2})_{T_2\in\mathcal{M}_{2,h}},
  (q_{T_1})_{T_1 \in \mathcal{M}_{1,h}},
  (q_{T_0})_{T_0 \in \mathcal{M}_{0,h}}
  \big),
  \\
  \Rcurl{h}\uvec{v}_{h}
  \coloneq \big(
  (\RRoly{T_3}\uvec{v}_{T_3}, \cRproj{\ell_{T_3}+1}{T_3}\bvec{v}_{\cvec{R},{T_3}}^\compl)_{T_3\in\mathcal{M}_{3,h}},
  (\bvec{v}_{\cvec{R},T_2},\cRproj{\ell_{T_2}+1}{T_2}\bvec{v}_{\cvec{R},T_2}^\compl)_{T_2\in\mathcal{M}_{2,h}},
  (v_{T_1}
  \big)_{T_1\in\mathcal{M}_{1,h}}).
\end{gathered}
\]

The complexes $(W_i,\partial_i)_i $ and $(\widehat{W}_i,\widehat{\partial}_i)_i$ for the application of Section~\ref{sec:stokes} along with the corresponding extension and reduction maps are summarized in the following diagram:
\begin{equation}\label{eq:ddr3d.sddr3d}
\begin{tikzpicture}[xscale=2.5, yscale=1.25, baseline={(WL2.base)}]
  \node at (-1,1) {DDR3d:};
  \node at (-1,-1) {SDDR3d:};

  \node (Xgrad) at (0,1) {$\Xgrad{k-1,k-1,k}{h}$};
  \node (Xcurl) at (1.5,1) {$\Xcurl{k,k,k}{h}$};
  \node (Xdiv) at (3,0) {$\Xdiv{h}$};
  \node (WL2) at (4.5,0) {$\XL{k}{h}$};
  
  \draw[->,>=latex] (Xgrad) -- (Xcurl) node[midway, above, font=\footnotesize]{$\Gddr{h}$};
  \draw[->,>=latex] (Xcurl) -- (Xdiv) node[midway, above, font=\footnotesize]{$\Cddr{h}$};
  \draw[->,>=latex] (Xdiv) -- (WL2) node[midway, above, font=\footnotesize]{$\Dh{h}$};
    
  \node (SXgrad) at (0,-1) {$\SXgrad{k}{h}$};
  \node (SXcurl) at (1.5,-1) {$\SXcurl{h}$};

  \draw[->,>=latex] (SXgrad) -- (SXcurl) node[midway, above, font=\footnotesize]{$\sGddr{h}$};
  \draw[->,>=latex] (SXcurl) -- (Xdiv) node[midway, below, font=\footnotesize]{$\sCddr{h}$};
  
  \draw [->,>=latex] (SXgrad) to [bend right=10] node[midway, right, font=\footnotesize] {$\Egrad{h}$} (Xgrad) ;
  \draw [->,>=latex,dashed] (Xgrad) to [bend right=10] node[midway, left, font=\footnotesize] {$\Rgrad{h}$} (SXgrad) ;

  \draw [->,>=latex] (SXcurl) to [bend right=10] node[midway, right, font=\footnotesize] {$\Ecurl{h}$} (Xcurl) ;
  \draw [->,>=latex,dashed] (Xcurl) to [bend right=10] node[midway, left, font=\footnotesize] {$\Rcurl{h}$} (SXcurl) ;

\end{tikzpicture}
\end{equation}
where $\sGddr{h}$ and $\sCddr{h}$ are given by \eqref{eq:comm.diag}.\\

\subsection{Cohomology of the serendipity DDR complexes}

We recall the following result from \cite{Di-Pietro.Droniou:23*1} (see, in particular, Lemmas~22 and~26 therein).
\begin{lemma}[Cohomology of the DDR and SDDR complexes]\label{lem:asm.SDDR}
  The two- and three-dimensional DDR and SDDR complexes, together with their extension and reduction operators, satisfy Assumption \ref{Isom.W.W}.
  In particular, this implies that both the cohomologies of the SDDR and DDR complexes are isomorphic to the cohomology of the corresponding continuous de Rham complex.  
\end{lemma}


\section{A serendipity rot-rot complex}\label{sec:rotrot}

We now turn to the first application of the general construction considering the following smoother variant of the two-dimensional de Rham complex:
\begin{equation}\label{eq:complex.rotrot}
  \begin{tikzpicture}[xscale=2,baseline={(H1head.base)}]
    \node (H1head) at (0,0) {$H^1(\Omega)$};
    \node (Hrotrot) at (1.5,0) {$\Hrotrot{\Omega}$};
    \node (H1tail) at (3,0) {$H^1(\Omega)$,};

    \draw[->,>=latex] (H1head) -- (Hrotrot) node[midway, above, font=\footnotesize]{$\GRAD$};
    \draw[->,>=latex] (Hrotrot) -- (H1tail) node[midway, above, font=\footnotesize]{$\ROT$};
  \end{tikzpicture}  
\end{equation}
where $\Omega \subset \mathbb{R}^2$ is a polygonal domain and, for a smooth enough vector-valued field $\bvec{v}$, $\ROT \bvec{v} \coloneq \DIV \bvec{v}^\perp$.

Diagram~\eqref{eq:abstract.diagram} specialized to the present case becomes
\begin{equation}\label{eq:construction:rot-rot}
  \begin{tikzpicture}[xscale=2, yscale=2.5]
    \node at (1,0,0) {Srot-rot:};
    \node at (1,0,-2) {rot-rot:};
    \node at (1,2,0) {SDDR2d:};
    \node at (1,2,-2) {DDR2d:};
    
    \node (SVser) at (2,0,0) {$\SVser{h}$};
    \node (SSigmaser) at (4,0,0) {$\SSigmaser{h}$};
    \node (SW) at (6,0,-1) {$\SW{h}$};

    \node (SV) at (2,0,-2) {$\SV{h}$};
    \node (SSigma) at (4,0,-2) {$\SSigma{h}$};

    \node (SXGrad) at (2,2,0) {$\SXgrad{k}{h}$};
    \node (SXRot) at (4,2,0) {$\SXrot{h}$};
    \node (Polyk) at (6,2,-1) {$\XL{k}{h}$};

    \node (XGrad) at (2,2,-2) {$\Xgrad{k-1,k}{h}$};
    \node (XRot) at (4,2,-2) {$\Xrot{k,k}{h}$};

    \draw [->,>=latex] (SVser) -- (SSigmaser) node[midway, below,font=\small] {$\suGh{h}$};
    \draw [->,>=latex] (SSigmaser) -- (SW) node[midway, below,font=\small] {$\suRh{h}$};
    \draw [->,>=latex] (SV) -- (SSigma) node[midway, above,font=\small] {$\uGh{h}$};
    \draw [->,>=latex] (SSigma) -- (SW) node[midway, above,font=\small] {$\uRh{h}$};
    \draw [->,>=latex] (SXGrad) -- (SXRot) node[midway, below,font=\small] {$\sGddr{h}$};
    \draw [->,>=latex] (SXRot) -- (Polyk) node[midway, below,font=\small] {$\sRddr{h}$};

    \draw [->,>=latex] (XGrad) -- (XRot) node[midway, above,font=\small] {$\Gddr{h}$};
    \draw [->,>=latex] (XRot) -- (Polyk) node[midway, above,font=\small] {$\rotddr{h}$};

    \draw [->,>=latex] (SVser) to [bend right=20] node[midway, right,font=\small] {$\EV{h}$} (SV) ;
    \draw [->,dashed] (SV) to [bend right=20] node[midway, left,font=\small] {$\RV{h}$} (SVser);
    \draw [->,>=latex] (SSigmaser) to [bend right=20] node[midway, right,font=\small] {$\ES{h}$} (SSigma) ;
    \draw [->,dashed] (SSigma) to [bend right=20] node[midway, left, font=\small] {$\RS{h}$} (SSigmaser);

    \draw [->,>=latex,dashed] (XGrad) to [bend right=20] node[midway, left, font=\footnotesize] {$\Rgrad{h}$} (SXGrad);
    \draw [->] (SXGrad) to [bend right=20] node[midway, right, font=\footnotesize] {$\Egrad{h}$} (XGrad);
    \draw [->,>=latex, dashed] (XRot) to [bend right=20] node[midway, left, font=\footnotesize] {$\Rrot{h}$} (SXRot);
    \draw [->] (SXRot) to [bend right=20] node[midway, right, font=\footnotesize] {$\Erot{h}$} (XRot);
    
    \draw [<->] (XGrad) -- (SV) node[pos=0.8, right,font=\small] {Id};
    \draw [->] (SXRot) to [bend left=20] node[pos=0.4, right,font=\footnotesize] {$\sInjrot$} (SSigmaser);
    
    \draw [->] (XRot) to [bend left=20] node[pos=0.6, right, font=\footnotesize] {$\Injrot$} (SSigma);    
    \draw [->] (Polyk) to [bend left=20] node[pos=0.6, right, font=\footnotesize] {$\InjW{h}$} (SW);
    
    \draw [->,dashed] (SW) to [bend left=20] node[pos=0.4, left, font=\footnotesize] {$\ResW{h}$} (Polyk);
    \draw [->,dashed] (SSigma) to [bend left=20] node[pos=0.4, left,font=\footnotesize]  {$\Resrot$} (XRot);    

    \draw [->,dashed] (SSigmaser) to [bend left=20] node[pos=0.6, left,font=\footnotesize]  {$\sResrot$} (SXRot);
    \draw [<->] (SXGrad) -- (SVser) node[pos=0.4, right,font=\small] {Id};
  \end{tikzpicture}
\end{equation}
The top horizontal portion of the above diagram corresponds to \eqref{eq:ddr2d.sddr2d}.
In the rest of this section we will provide a precise definition of the other spaces and operators that appear in it and, using the abstract framework of Section~\ref{sec:gen}, show that all the complexes involved have isomorphic cohomologies.

\subsection{Discrete rot-rot complex}

A discrete counterpart of the complex \eqref{eq:complex.rotrot} was developed in \cite{Di-Pietro:23}.
We briefly recall its construction here.
We define the discrete head $H^1(\Omega)$, $\Hrotrot{\Omega}$, and tail $H^1(\Omega)$ spaces as follows:
\[
\SV{h} \coloneq\Xgrad{k-1,k}{h},\quad
\SSigma{h} \coloneq\Xrot{k,k}{h}\times\left(\bigtimes_{T_1\in\mathcal{M}_{1,h}}\Poly{k-1}(T_1)\times\Real^{\mathcal{M}_{0,h}} \right),\quad
\SW{h} \coloneq\Xgrad{k,k}{h}.
\]
The discrete gradient and rotor are respectively such that, for all $\underline{q}_h\in \SV{h}$ and all $\uvec{v}_h=\big(\uvec{v}_{w,h}, \underline v_{\compl,h} \big)\in\SSigma{h}$,
\begin{align}\label{eq:uGh}
  \uGh{h}~\underline{q}_h
  &\coloneq\big(
  \Gddr{h}\underline{q}_h,\underline{0}
  \big),
  \\ \label{eq:uRh}
  \uRh{h}\uvec{v}_h
  &\coloneq\big(
  \rotddr{h}\uvec{v}_{w,h}, \underline{v}_{\compl,h}
  \big).
\end{align}
The discrete counterpart of \eqref{eq:complex.rotrot} is then given by:
\begin{equation}\label{eq:d.rotrot.complex}
    \begin{tikzpicture}[xscale=2,baseline={(H1head.base)}]
    \node at (-1,0) {rot-rot:};
    
    \node (H1head) at (0,0) {$\SV{h}$};
    \node (Xrotrot) at (1.5,0) {$\SSigma{h}$};
    \node (H1tail) at (3,0) {$\SW{h}$.};

    \draw[->,>=latex] (H1head) -- (Xrotrot) node[midway, above, font=\footnotesize]{$\uGh{h}$};
    \draw[->,>=latex] (Xrotrot) -- (H1tail) node[midway, above, font=\footnotesize]{$\uRh{h}$};
  \end{tikzpicture}
\end{equation}

\subsection{Extension and reduction maps between the two-dimensional DDR and rot-rot complexes}\label{eq:rot-rot:extension.reduction}

In order to apply the construction of Definition~\ref{def:serendipity.complex} to define and characterize a serendipity version of this complex, we need extension and reduction maps between the two-dimensional DDR complex \eqref{eq:glob.2D-ddr.complex} and the discrete rot-rot complex \eqref{eq:d.rotrot.complex}.
Noticing that
\[
\SW{h}=\XL{k}{h}\times\left(\bigtimes_{T_1\in\mathcal{M}_{1,h}}\Poly{k-1}(T_1)\times\Real^{\mathcal{M}_{0,h}} \right),
\]
the spaces $\Xrot{k,k}{h}$ and $\XL{k}{h}$ inject respectively into $\SSigma{h}$ and $\SW{h}$ trough the extension map such that, for all $\uvec{v}_{w,h}\in\Xrot{k,k}{h}$ and all $q_h\in \XL{k}{h}$,
\begin{equation}\label{eq:Injrot.InjW}
  \text{
    $\Injrot\uvec{v}_{w,h} 
    \coloneq\big( \uvec{v}_{w,h}, \underline{0}\big)$
    and 
    $\InjW{h}q_h
    \coloneq \big( q_h,\underline{0}\big)$.
  }
\end{equation}
We also define the reduction map such that, for all $\uvec{v}_h=(\uvec{v}_{w,h},\underline{v}_{\compl,h})\in\SSigma{h}$ and all $\underline{q}_h=(q_h,\underline{q}_{\compl,h})\in\SW{h}$,
\begin{equation}\label{eq:Resrot.ResW}
  \text{%
    $\Resrot\uvec{v}_h
    \coloneq\uvec{v}_{w,h}$
    and
    $\ResW{h}\underline {q}_h
    \coloneq q_h$.
  }
\end{equation}
The decomposition of Lemma \ref{def.v} clearly holds by definition, so we have
\[
\text{%
  $\SSigma{h}=\Injrot \Xrot{k,k}{h} \oplus \Kernel \Resrot$
  and $\SW{h}=\InjW{h} \XL{k}{h}\oplus \Kernel \ResW{h}$.
}
\]

\begin{theorem}[Properties of the extension and reduction maps between the DDR2d and rot-rot complexes]\label{thm.hom.rotrot}  
The maps defined by \eqref{eq:Injrot.InjW} and \eqref{eq:Resrot.ResW} satisfy Assumption \ref{Isom.W.V}, i.e.,  
  \begin{enumerate}[label=\textbf{(B{\arabic*})}]
  \item For all $\uvec{v}_{w,h}\in\Xrot{k,k}{h}$ and all $ q_h\in \XL{k}{h}$,
    \begin{equation}\label{eq:isomorphism.rot.W}
      \text{
        $\Resrot\Injrot\uvec{v}_{w,h} = \uvec{v}_{w,h}$
        and
        $\ResW{h}\InjW{h}q_h = q_h$.
      }
    \end{equation}
    
  \item For all $\uvec{v}_h=(\uvec{v}_{w,h},\underline{v}_{\compl,h})\in\Kernel\uRh{h}$,
    \begin{equation}\label{eq:complex.rot}
      \Injrot\Resrot\uvec{v}_h-\uvec{v}_h\in\Image(\uGh{h}).
    \end{equation}
    
  \item For all $\underline{q}_h\in \SV{h}$, all $\uvec{v}_h\in\SSigma{h}$, and all $\uvec{v}_{w,h}\in\Xrot{k,k}{h}$, it holds
    \begin{alignat}{4}\label{eq:cochain.rot}
      \Resrot\uGh{h}\underline{q}_h &= \Gddr{h} \underline{q}_h,&\qquad
      \Injrot\Gddr{h} \underline{q}_h &= \uGh{h}\underline{q}_h,
      \\ \label{eq:cochain.W}
      \ResW{h}\uRh{h} \uvec{v}_h &= \rotddr{h} \Resrot\uvec{v}_h,&\qquad
      \InjW{h} \rotddr{h} \uvec{v}_{w,h} &= \uRh{h}\Injrot\uvec{v}_{w,h}.
    \end{alignat}
  \end{enumerate}
  It then follows from Remark~\ref{rem:iso.cohom} that the two-dimensional DDR complex \eqref{eq:glob.2D-ddr.complex} and the rot-rot complex \eqref{eq:d.rotrot.complex} have isomorphic cohomologies.
\end{theorem}

\begin{proof}
  \underline{(i) \emph{Proof of \eqref{eq:isomorphism.rot.W}.}}
  For all $\uvec{v}_{w,h}\in\Xrot{k,k}{h}$,
  $\Resrot\Injrot\uvec{v}_{w,h}\overset{\eqref{eq:Injrot.InjW}}=\Resrot(\uvec{v}_{w,h},\underline{0})\overset{\eqref{eq:Resrot.ResW}}=\uvec{v}_{w,h}$ and, for all $q_h\in \XL{k}{h}$,
  $\ResW{h}\InjW{h}q_h\overset{\eqref{eq:Injrot.InjW}}=\ResW{h}(q_h,\underline{0})\overset{\eqref{eq:Resrot.ResW}}=q_h$.
  \medskip\\
  \underline{(ii) \emph{Proof of \eqref{eq:complex.rot}.}}
  Let $\uvec{v}_h\in\Kernel\uRh{h}$. Using the definition \eqref{eq:uRh} of $\uRh{h}$, we obtain that $\uvec{v}_h=(\uvec{v}_{w,h},\underline{0})$, so $\Injrot\Resrot\uvec{v}_h-\uvec{v}_h=\uvec{0}=\uGh{h}\underline{0}.$
  \medskip\\
  \underline{(iii) \emph{Proof of \eqref{eq:cochain.rot}.}}
  For all $\underline{q}_h\in\SV{h}$, we have
  $\Resrot\uGh{h}\underline{q}_h \overset{\eqref{eq:uGh}}= \Resrot(\Gddr{h}\underline{q}_h,\underline{0}) \overset{\eqref{eq:Resrot.ResW}}= \Gddr{h}\underline{q}_h$
  and
  $\Injrot\Gddr{h}\underline{q}_h \overset{\eqref{eq:Injrot.InjW}}= (\Gddr{h}\underline{q}_h,\underline{0}) \overset{\eqref{eq:uGh}}= \uGh{h}\underline{q}_h$.
  \medskip\\
  \underline{(vi) \emph{Proof of \eqref{eq:cochain.W}.}}
  For all $\uvec{v}_h=(\uvec{v}_{w,h},\underline{v}_{\compl,h})\in\SSigma{h}$,
  \[
  \ResW{h}\uRh{h}(\uvec{v}_{w,h},\underline{v}_{\compl,h})
  \overset{\eqref{eq:uRh}}= \ResW{h}(\rotddr{h}\uvec{v}_{w,h},\underline{v}_{\compl,h})
  \overset{\eqref{eq:Resrot.ResW}}= \rotddr{h}\uvec{v}_{w,h}
  \overset{\eqref{eq:Resrot.ResW}}= \rotddr{h}\Resrot\uvec{v}_h
  \]
  and, for all $\uvec{v}_{w,h}\in\Xrot{k,k}{h}$,
  \[
  \InjW{h}\rotddr{h}\uvec{v}_{w,h}
  \overset{\eqref{eq:Injrot.InjW}}= (\rotddr{h}\uvec{v}_{w,h},\underline{0})
  \overset{\eqref{eq:uRh}}= \uRh{h}(\uvec{v}_{w,h},\underline{0})
  \overset{\eqref{eq:Injrot.InjW}}= \uRh{h}\Injrot\uvec{v}_{w,h}.\qedhere
  \]
\end{proof}


\subsection{Serendipity rot-rot complex and homological properties}\label{sec:discrete.rot-rotcomplex:serendipity}

Lemma \ref{lem:asm.SDDR} and Theorem \ref{thm.hom.rotrot} ensure that the SDDR and rot-rot complexes satisfy Assumptions~\ref{Isom.W.W} and~\ref{Isom.W.V}. 
We are now in a position to apply the construction \eqref{eq:construction} to the rot-rot complex in order to derive its serendipity version and characterize its cohomology.

\subsubsection{Serendipity spaces and operators}\label{sec:discrete.rot-rotcomplex:serendipity:spaces}

Recalling \eqref{eq:def.v^}, the serendipity version of spaces $\SV{h}$ and $\SSigma{h}$ can be written as follows:
\begin{equation}\label{eq:S.D.Spaces}
  \begin{gathered}
    \SVser{h} \coloneq
    \SXgrad{k}{h}
    \\
    \SSigmaser{h}
    \coloneq
    \SXrot{h}\times \Kernel \Resrot
    \cong
    \SXrot{h}\times\left(
    \bigtimes_{T_1\in\mathcal{M}_{1,h}}\Poly{k-1}(T_1)\times\Real^{\mathcal{M}_{0,h}}
    \right).
  \end{gathered}
\end{equation}
Accounting for the isomorphism in \eqref{eq:S.D.Spaces}, we write a generic element $\suvec{v}_h$ of $\SSigmaser{h}$ as $\suvec{v}_h=\big(\suvec{v}_{w,h},\underline{v}_{\compl,h}\big)$ with $\suvec{v}_{w,h}\in\SXrot{h}$ and $\underline{v}_{\compl,h}$ such that $(\suvec{0},\underline{v}_{\compl,h})\in\Kernel\Resrot$.
We define the extension of $\SXrot{h}$ into $\SSigmaser{h}$ according to \eqref{eq:def.sExt}:
\[
\sInjrot\suvec{v}_{w,h} \coloneq \big( \suvec{v}_{w,h}, \underline{0} \big).
\]
The reduction is given by \eqref{eq:def.sRed}:
\[
\sResrot(\suvec{v}_{w,h},\underline{v}_{\compl,h}) \coloneq \suvec{v}_{w,h}.
\]

The reduction operators $\RV{h}:\SV{h}\to\SVser{h}$ and $\RS{h}:\SSigma{h}\to\SSigmaser{h}$ are defined using \eqref{eq:def.ReV} and accounting for the isomorphism \eqref{eq:S.D.Spaces}:
For all $\underline{q}_h\in\SV{h}$ and all $\uvec{v}_h\in\SSigma{h}$,
\[
\text{%
  $\RV{h}\underline{q}_h
  \coloneq
  \Rgrad{h}\underline{q}_h$
  and
  $\RS{h}\uvec{v}_h
  \coloneq
  \Big(\Rrot{h}\Resrot\uvec{v}_h,\underline{v}_{\compl,h}\Big)$,
}
\]
with $\Rgrad{h}$ and $\Rrot{h}$ respectively defined according to \eqref{eq:Rgrad.T2} and \eqref{eq:Rrot.T2}.

Finally, using \eqref{eq:def.EX}, the extension operators $\EV{h}:\SVser{h}\to\SV{h}$ and $\ES{h}:\SSigmaser{h}\to\SSigma{h}$ are such that, 
for all $\underline{\widehat{q}}_h\in\SVser{h}$ and all $(\suvec{v}_{w,h},\underline{v}_{\compl,h})\in\SSigmaser{h}$,
\[
\text{
  $\EV{h}\underline{\widehat{q}}_h
  \coloneq
  \Egrad{h} \underline{\widehat{q}}_h$
  and
  $\ES{h}\suvec{v}_h
  \coloneq
  \Injrot\Erot{h}\suvec{v}_{w,h} + (\suvec{0},\underline{v}_{\compl,h})$,
}
\]
with $\Egrad{h}$ and $\Erot{h}$ respectively defined according to \eqref{eq:Egrad.T2} and \eqref{eq:Erot.T2}.

Using \eqref{eq:def.d}, the serendipity discrete differential operators are such that, for all $(\underline{\widehat{q}}_h,\suvec{v}_h)\in\SVser{h}\times\SSigmaser{h}$:
\[
\begin{gathered}
  \suGh{h} \underline{\widehat{q}}_h
  \coloneq\big(\sGddr{h}\underline{\widehat{q}}_h,\underline{0}\big),
  \\
  \suRh{h}\suvec{v}_h
  \coloneq \big(\sRddr{h}\suvec{v}_{w,h},\uRh{h}(\suvec{0},\underline{v}_{\compl,h})\big)\overset{\eqref{eq:uRh},\,\eqref{eq:S.D.Spaces}}=\big(\sRddr{h}\suvec{v}_{w,h},\underline{v}_{\compl,h}\big).
\end{gathered}
\]

\subsubsection{Serendipity rot-rot complex and isomorphism in cohomology}

The serendipity rot-rot complex is given by:
\begin{equation}\label{eq:d.srotrot.complex}  
  \begin{tikzpicture}[xscale=2, baseline={(H1head.base)}]
    \node at (-1,0) {Srot-rot:};
    
    \node (H1head) at (0,0) {$\SVser{h}$};
    \node (Xrotrot) at (1.5,0) {$\SSigmaser{h}$};
    \node (H1tail) at (3,0) {$\SW{h}$.};

    \draw[->,>=latex] (H1head) -- (Xrotrot) node[midway, above, font=\footnotesize]{$\suGh{h}$};
    \draw[->,>=latex] (Xrotrot) -- (H1tail) node[midway, above, font=\footnotesize]{$\suRh{h}$};
  \end{tikzpicture}
\end{equation}

\begin{theorem}[Homological properties of the complexes in \eqref{eq:construction:rot-rot}]\label{thm:homological.properties}  
All the complexes in the diagram \eqref{eq:construction:rot-rot} have cohomologies that are isomorphic to the cohomology of the continuous de Rham complex.
\end{theorem}

\begin{proof}
  Lemma \ref{lem:asm.SDDR} and Theorem \ref{thm.hom.rotrot} ensure that 
  Assumptions \ref{Isom.W.W} and \ref{Isom.W.V} are satisfied.
  We can therefore invoke Corollary \ref{Isom.V.V} to infer that the cohomology of the Srot-rot complex \eqref{eq:d.srotrot.complex} is isomorphic to the cohomology of the rot-rot complex \eqref{eq:d.rotrot.complex}, of the DDR2d complex \eqref{eq:glob.2D-ddr.complex}, and, therefore, of the continuous de Rham complex.  
\end{proof}

\subsection{Numerical examples}

In order to show the effect of serendipity DOF reduction, we consider the quad-rot problem of \cite[Section~5.2]{Di-Pietro:23} and compare the results obtained using the original and serendipity spaces in terms of error versus dimension of the linear system (after elimination of Dirichlet DOFs).
The errors are defined as the difference between the solution of the numerical scheme and the interpolate of the exact solution.
Specifically, denoting respectively by $(\uvec{u}_h, \underline{p}_h)$ and $(\suvec{u}_h, \underline{\widehat{p}}_h)$ the numerical solutions obtained using standard and serendipity spaces, we set
\[
\begin{alignedat}{3}
  \uvec{e}_h &\coloneq \uvec{u}_h - \ISigma{h} \bvec{u}
  &\qquad \underline{\varepsilon}_h &\coloneq 
  \underline{p}_h - \IV{h} p,
  \\
  \suvec{e}_h &\coloneq \suvec{u}_h - \sISigma{h} \bvec{u}
  &\qquad \underline{\widehat \varepsilon}_h &\coloneq \underline{\widehat p}_h - \sIV{h} p,  
\end{alignedat}
\]
where $\IV{h}$, $\sIV{h}$, $\ISigma{h}$, and $\sISigma{h}$ respectively denote the interpolators on $\SV{h}$, $\SVser{h}$, $\SSigma{h}$, and $\SSigmaser{h}$.
The errors are measured by $L^2$-like operator norms defined in the spirit of \cite[Section 4.4]{Di-Pietro.Droniou:23*2} and, consistently with \cite{Di-Pietro:23}, respectively denoted by $\norm{V,h}{\cdot}$ for $\SV{h}$ and $\SVser{h}$
and 
$\norm{\bvec{\Sigma},h}{\cdot}$ for $\SSigma{h}$ and $\SSigmaser{h}$ (we do not distinguish the notation for the norms on the standard and serendipity spaces, as they have formally the same expression and the exact meaning is made clear by the argument).
On the latter spaces, we additionally consider the norm $\norm{\VROT\ROT,h}{\cdot}$, an $L^2$-like norm of the discrete rot-rot operator defined as in \cite[Eq. (4.29)]{Di-Pietro:23}.
The problem data, meshes, and polynomial degrees are exactly the same as in the above reference, so we do not repeat these details here, while the number of edges $\eta_{T_1}$ for each edge $T_1 \in \mathcal{M}_{h,1}$ is chosen the same way as in \cite{Di-Pietro.Droniou:23*1}.
The various error measures displayed in Figures \ref{fig:cartesian.dof}--\ref{fig:hexagonal.dof} show that a given precision is invariably obtained with fewer DOFs using serendipity spaces, the more so the higher the degree.
A comparison in terms of error versus meshsize $h$, not reported here for the sake of conciseness, shows that the serendipity and non-serendipity schemes yield essentially the same solution for a given mesh and polynomial degree, with visible differences only for the pressure errors
$\norm{V,h}{\underline{\varepsilon}_h}$,
$\norm{V,h}{\underline{\widehat \varepsilon}_h}$,
$\norm{\svec{\Sigma},h}{\uGh{h} \underline{\varepsilon}_h}$,
and $\norm{\svec{\Sigma},h}{\uGh{h} \underline{\widehat \varepsilon}_h}$ for $k=3$.


\begin{figure}\centering
  {
    \hypersetup{hidelinks}
    \ref{leg:cartesian}
  }
  \vspace{0.5cm}\\
  \begin{minipage}{0.475\linewidth}\centering
    \begin{tikzpicture}[scale=0.90]
      \begin{loglogaxis}[
          legend columns=3,
          legend to name=leg:cartesian
        ]
        \addplot [mark=square*, black] table[x=DimLinSys,y=ErrUL2] {dat/ddr-quadrot/cart_k1/data_rates.dat};
        \addplot [mark=square*, red] table[x=DimLinSys,y=ErrUL2] {dat/ddr-quadrot/cart_k2/data_rates.dat};
        \addplot [mark=square*, cyan] table[x=DimLinSys,y=ErrUL2] {dat/ddr-quadrot/cart_k3/data_rates.dat};
        \addplot [mark=square*, mark options=solid, black, dashed] table[x=DimLinSys,y=ErrUL2] {dat/sddr-quadrot/cart_k1/data_rates.dat};
        \addplot [mark=square*, mark options=solid, red, dashed] table[x=DimLinSys,y=ErrUL2] {dat/sddr-quadrot/cart_k2/data_rates.dat};
        \addplot [mark=square*, mark options=solid, cyan, dashed] table[x=DimLinSys,y=ErrUL2] {dat/sddr-quadrot/cart_k3/data_rates.dat};

        \legend{$k=1$ (rot-rot), $k=2$ (rot-rot), $k=3$ (rot-rot), $k=1$ (Srot-rot), $k=2$ (Srot-rot), $k=3$ (Srot-rot)};
      \end{loglogaxis}
    \end{tikzpicture}
    \subcaption{$\norm{\bvec{\Sigma},h}{\bvec{e}_h}$ and $\norm{\bvec{\Sigma},h}{\suvec{e}_h}$}
  \end{minipage}
  \begin{minipage}{0.475\linewidth}\centering
    \begin{tikzpicture}[scale=0.90]
      \begin{loglogaxis}[legend columns=3]
        \addplot [mark=square*, black] table[x=DimLinSys,y=ErrURotRot] {dat/ddr-quadrot/cart_k1/data_rates.dat};
        \addplot [mark=square*, red] table[x=DimLinSys,y=ErrURotRot] {dat/ddr-quadrot/cart_k2/data_rates.dat};
        \addplot [mark=square*, cyan] table[x=DimLinSys,y=ErrURotRot] {dat/ddr-quadrot/cart_k3/data_rates.dat};
        \addplot [mark=square*, mark options=solid, black, dashed] table[x=DimLinSys,y=ErrURotRot] {dat/sddr-quadrot/cart_k1/data_rates.dat};
        \addplot [mark=square*, mark options=solid, red, dashed] table[x=DimLinSys,y=ErrURotRot] {dat/sddr-quadrot/cart_k2/data_rates.dat};
        \addplot [mark=square*, mark options=solid, cyan, dashed] table[x=DimLinSys,y=ErrURotRot] {dat/sddr-quadrot/cart_k3/data_rates.dat};

      \end{loglogaxis}
    \end{tikzpicture}
    \subcaption{$\norm{\VROT\ROT,h}{\uvec{e}_h}$ and $\norm{\VROT\ROT,h}{\suvec{e}_h}$}
  \end{minipage}
  \vspace{0.5cm}\\
  \begin{minipage}{0.475\linewidth}\centering
    \begin{tikzpicture}[scale=0.90]
      \begin{loglogaxis}
        \addplot [mark=square*, black] table[x=DimLinSys,y=ErrPL2] {dat/ddr-quadrot/cart_k1/data_rates.dat};
        \addplot [mark=square*, red] table[x=DimLinSys,y=ErrPL2] {dat/ddr-quadrot/cart_k2/data_rates.dat};
        \addplot [mark=square*, cyan] table[x=DimLinSys,y=ErrPL2] {dat/ddr-quadrot/cart_k3/data_rates.dat};
        \addplot [mark=square*, mark options=solid, black, dashed] table[x=DimLinSys,y=ErrPL2] {dat/sddr-quadrot/cart_k1/data_rates.dat};
        \addplot [mark=square*, mark options=solid, red, dashed] table[x=DimLinSys,y=ErrPL2] {dat/sddr-quadrot/cart_k2/data_rates.dat};
        \addplot [mark=square*, mark options=solid, cyan, dashed] table[x=DimLinSys,y=ErrPL2] {dat/sddr-quadrot/cart_k3/data_rates.dat};
      \end{loglogaxis}
    \end{tikzpicture}
    \subcaption{$\norm{V,h}{\underline{\varepsilon}_h}$ and $\norm{V,h}{\underline{\widehat{\varepsilon}}_h}$}
  \end{minipage}
  \begin{minipage}{0.475\linewidth}\centering
    \begin{tikzpicture}[scale=0.90]
      \begin{loglogaxis}
        \addplot [mark=square*, black] table[x=DimLinSys,y=ErrPGrad] {dat/ddr-quadrot/cart_k1/data_rates.dat};
        \addplot [mark=square*, red] table[x=DimLinSys,y=ErrPGrad] {dat/ddr-quadrot/cart_k2/data_rates.dat};
        \addplot [mark=square*, cyan] table[x=DimLinSys,y=ErrPGrad] {dat/ddr-quadrot/cart_k3/data_rates.dat};
        \addplot [mark=square*, mark options=solid, black, dashed] table[x=DimLinSys,y=ErrPGrad] {dat/sddr-quadrot/cart_k1/data_rates.dat};
        \addplot [mark=square*, mark options=solid, red, dashed] table[x=DimLinSys,y=ErrPGrad] {dat/sddr-quadrot/cart_k2/data_rates.dat};
        \addplot [mark=square*, mark options=solid, cyan, dashed] table[x=DimLinSys,y=ErrPGrad] {dat/sddr-quadrot/cart_k3/data_rates.dat};
      \end{loglogaxis}
    \end{tikzpicture}
    \subcaption{${\norm{\bvec{\Sigma},h}{\uGh{h} \underline{\varepsilon}_h}}$ and ${\norm{\bvec{\Sigma},h}{\suGh{h}\underline{\widehat{\varepsilon}}_h}}$}
  \end{minipage}
  \caption{Errors norm vs.\ linear system size using the standard (continuous lines) and serendipity spaces (dashed lines)  to solve the quad-rot problem of \cite[Section~5.2]{Di-Pietro:23} on the Cartesian orthogonal mesh family.\label{fig:cartesian.dof}}
\end{figure}


\begin{figure}\centering
  {
    \hypersetup{hidelinks}
    \ref{leg:triangular}
  }
  \vspace{0.5cm}\\
  \begin{minipage}{0.475\linewidth}\centering
    \begin{tikzpicture}[scale=0.90]
      \begin{loglogaxis}[
          legend columns=3,
          legend to name=leg:triangular
        ]
        \addplot [mark=square*, black] table[x=DimLinSys,y=ErrUL2] {dat/ddr-quadrot/tri_k1/data_rates.dat};
        \addplot [mark=square*, red] table[x=DimLinSys,y=ErrUL2] {dat/ddr-quadrot/tri_k2/data_rates.dat};
        \addplot [mark=square*, cyan] table[x=DimLinSys,y=ErrUL2] {dat/ddr-quadrot/tri_k3/data_rates.dat};
        \addplot [mark=square*, mark options=solid, black, dashed] table[x=DimLinSys,y=ErrUL2] {dat/sddr-quadrot/tri_k1/data_rates.dat};
        \addplot [mark=square*, mark options=solid, red, dashed] table[x=DimLinSys,y=ErrUL2] {dat/sddr-quadrot/tri_k2/data_rates.dat};
        \addplot [mark=square*, mark options=solid, cyan, dashed] table[x=DimLinSys,y=ErrUL2] {dat/sddr-quadrot/tri_k3/data_rates.dat};
        
        \legend{$k=1$ (rot-rot), $k=2$ (rot-rot), $k=3$ (rot-rot), $k=1$ (Srot-rot), $k=2$ (Srot-rot), $k=3$ (Srot-rot)};
      \end{loglogaxis}
    \end{tikzpicture}
    \subcaption{$\norm{\bvec{\Sigma},h}{\uvec{e}_h}$ and $\norm{\bvec{\Sigma},h}{\suvec{e}_h}$}
  \end{minipage}
  \begin{minipage}{0.475\linewidth}\centering
    \begin{tikzpicture}[scale=0.90]
      \begin{loglogaxis}[legend columns=3]
        \addplot [mark=square*, black] table[x=DimLinSys,y=ErrURotRot] {dat/ddr-quadrot/tri_k1/data_rates.dat};
        \addplot [mark=square*, red] table[x=DimLinSys,y=ErrURotRot] {dat/ddr-quadrot/tri_k2/data_rates.dat};
        \addplot [mark=square*, cyan] table[x=DimLinSys,y=ErrURotRot] {dat/ddr-quadrot/tri_k3/data_rates.dat};
        \addplot [mark=square*, mark options=solid, black, dashed] table[x=DimLinSys,y=ErrURotRot] {dat/sddr-quadrot/tri_k1/data_rates.dat};
        \addplot [mark=square*, mark options=solid, red, dashed] table[x=DimLinSys,y=ErrURotRot] {dat/sddr-quadrot/tri_k2/data_rates.dat};
        \addplot [mark=square*, mark options=solid, cyan, dashed] table[x=DimLinSys,y=ErrURotRot] {dat/sddr-quadrot/tri_k3/data_rates.dat};
      \end{loglogaxis}
    \end{tikzpicture}
    \subcaption{$\norm{\VROT\ROT,h}{\uvec{e}_h}$ and $\norm{\VROT\ROT,h}{\suvec{e}_h}$}
  \end{minipage}
  \vspace{0.5cm}\\
  \begin{minipage}{0.475\linewidth}\centering
    \begin{tikzpicture}[scale=0.90]
      \begin{loglogaxis}
        \addplot [mark=square*, black] table[x=DimLinSys,y=ErrPL2] {dat/ddr-quadrot/tri_k1/data_rates.dat};
        \addplot [mark=square*, red] table[x=DimLinSys,y=ErrPL2] {dat/ddr-quadrot/tri_k2/data_rates.dat};
        \addplot [mark=square*, cyan] table[x=DimLinSys,y=ErrPL2] {dat/ddr-quadrot/tri_k3/data_rates.dat};
        \addplot [mark=square*, mark options=solid, black, dashed] table[x=DimLinSys,y=ErrPL2] {dat/sddr-quadrot/tri_k1/data_rates.dat};
        \addplot [mark=square*, mark options=solid, red, dashed] table[x=DimLinSys,y=ErrPL2] {dat/sddr-quadrot/tri_k2/data_rates.dat};
        \addplot [mark=square*, mark options=solid, cyan, dashed] table[x=DimLinSys,y=ErrPL2] {dat/sddr-quadrot/tri_k3/data_rates.dat};
      \end{loglogaxis}
    \end{tikzpicture}
    \subcaption{$\norm{V,h}{\underline{\varepsilon}_h}$ and $\norm{V,h}{\underline{\widehat \varepsilon}_h} $}
  \end{minipage}
  \begin{minipage}{0.475\linewidth}\centering
    \begin{tikzpicture}[scale=0.90]
      \begin{loglogaxis}
        \addplot [mark=square*, black] table[x=DimLinSys,y=ErrPGrad] {dat/ddr-quadrot/tri_k1/data_rates.dat};
        \addplot [mark=square*, red] table[x=DimLinSys,y=ErrPGrad] {dat/ddr-quadrot/tri_k2/data_rates.dat};
        \addplot [mark=square*, cyan] table[x=DimLinSys,y=ErrPGrad] {dat/ddr-quadrot/tri_k3/data_rates.dat};
        \addplot [mark=square*, mark options=solid, black, dashed] table[x=DimLinSys,y=ErrPGrad] {dat/sddr-quadrot/tri_k1/data_rates.dat};
        \addplot [mark=square*, mark options=solid, red, dashed] table[x=DimLinSys,y=ErrPGrad] {dat/sddr-quadrot/tri_k2/data_rates.dat};
        \addplot [mark=square*, mark options=solid, cyan, dashed] table[x=DimLinSys,y=ErrPGrad] {dat/sddr-quadrot/tri_k3/data_rates.dat};
      \end{loglogaxis}
    \end{tikzpicture}
    \subcaption{${\norm{\bvec{\Sigma},h}{\uGh{h} \underline{\varepsilon}_h}}$ and ${\norm{\bvec{\Sigma},h}{\suGh{h}\underline{\widehat \varepsilon}_h}}$}
  \end{minipage}
  \caption{Errors norm vs.\ linear system size using the standard (continuous lines) and serendipity spaces (dashed lines)  to solve the quad-rot problem of \cite[Section~5.2]{Di-Pietro:23} on the triangular mesh family.\label{fig:triangular.dof}}
\end{figure}


\begin{figure}\centering
  {
    \hypersetup{hidelinks}
    \ref{leg:hexagonal}
  }
  \vspace{0.5cm}\\
  \begin{minipage}{0.475\linewidth}\centering
    \begin{tikzpicture}[scale=0.90]
      \begin{loglogaxis}[
          legend columns=3,
          legend to name=leg:hexagonal
        ]
        \addplot [mark=square*, black] table[x=DimLinSys,y=ErrUL2] {dat/ddr-quadrot/hexa_k1/data_rates.dat};
        \addplot [mark=square*, red] table[x=DimLinSys,y=ErrUL2] {dat/ddr-quadrot/hexa_k2/data_rates.dat};
        \addplot [mark=square*, cyan] table[x=DimLinSys,y=ErrUL2] {dat/ddr-quadrot/hexa_k3/data_rates.dat};
        \addplot [mark=square*, mark options=solid, black, dashed] table[x=DimLinSys,y=ErrUL2] {dat/sddr-quadrot/hexa_k1/data_rates.dat};
        \addplot [mark=square*, mark options=solid, red, dashed] table[x=DimLinSys,y=ErrUL2] {dat/sddr-quadrot/hexa_k2/data_rates.dat};
        \addplot [mark=square*, mark options=solid, cyan, dashed] table[x=DimLinSys,y=ErrUL2] {dat/sddr-quadrot/hexa_k3/data_rates.dat};
        \legend{$k=1$ (rot-rot), $k=2$ (rot-rot), $k=3$ (rot-rot), $k=1$ (Srot-rot), $k=2$ (Srot-rot), $k=3$ (Srot-rot)};
      \end{loglogaxis}
    \end{tikzpicture}
    \subcaption{$\norm{\bvec{\Sigma},h}{\uvec{e}_h}$ and $\norm{\bvec{\Sigma},h}{\suvec{e}_h}$}
  \end{minipage}
  \begin{minipage}{0.475\linewidth}\centering
    \begin{tikzpicture}[scale=0.90]
      \begin{loglogaxis}[legend columns=3]
        \addplot [mark=square*, black] table[x=DimLinSys,y=ErrURotRot] {dat/ddr-quadrot/hexa_k1/data_rates.dat};
        \addplot [mark=square*, red] table[x=DimLinSys,y=ErrURotRot] {dat/ddr-quadrot/hexa_k2/data_rates.dat};
        \addplot [mark=square*, cyan] table[x=DimLinSys,y=ErrURotRot] {dat/ddr-quadrot/hexa_k3/data_rates.dat};
        \addplot [mark=square*, mark options=solid, black, dashed] table[x=DimLinSys,y=ErrURotRot] {dat/sddr-quadrot/hexa_k1/data_rates.dat};
        \addplot [mark=square*, mark options=solid, red, dashed] table[x=DimLinSys,y=ErrURotRot] {dat/sddr-quadrot/hexa_k2/data_rates.dat};
        \addplot [mark=square*, mark options=solid, cyan, dashed] table[x=DimLinSys,y=ErrURotRot] {dat/sddr-quadrot/hexa_k3/data_rates.dat};
      \end{loglogaxis}
    \end{tikzpicture}
    \subcaption{$\norm{\VROT\ROT,h}{\uvec{e}_h}$ and $\norm{\VROT\ROT,h}{\suvec{e}_h}$}
  \end{minipage}
  \vspace{0.5cm}\\
  \begin{minipage}{0.475\linewidth}\centering
    \begin{tikzpicture}[scale=0.90]
      \begin{loglogaxis}
        \addplot [mark=square*, black] table[x=DimLinSys,y=ErrPL2] {dat/ddr-quadrot/hexa_k1/data_rates.dat};
        \addplot [mark=square*, red] table[x=DimLinSys,y=ErrPL2] {dat/ddr-quadrot/hexa_k2/data_rates.dat};
        \addplot [mark=square*, cyan] table[x=DimLinSys,y=ErrPL2] {dat/ddr-quadrot/hexa_k3/data_rates.dat};
        \addplot [mark=square*, mark options=solid, black, dashed] table[x=DimLinSys,y=ErrPL2] {dat/sddr-quadrot/hexa_k1/data_rates.dat};
        \addplot [mark=square*, mark options=solid, red, dashed] table[x=DimLinSys,y=ErrPL2] {dat/sddr-quadrot/hexa_k2/data_rates.dat};
        \addplot [mark=square*, mark options=solid, cyan, dashed] table[x=DimLinSys,y=ErrPL2] {dat/sddr-quadrot/hexa_k3/data_rates.dat};
      \end{loglogaxis}
    \end{tikzpicture}
    \subcaption{$\norm{V,h}{\underline{\varepsilon}_h}$ and $\norm{V,h}{\underline{\widehat \varepsilon}_h}$}
  \end{minipage}
  \begin{minipage}{0.475\linewidth}\centering
    \begin{tikzpicture}[scale=0.90]
      \begin{loglogaxis}
        \addplot [mark=square*, black] table[x=DimLinSys,y=ErrPGrad] {dat/ddr-quadrot/hexa_k1/data_rates.dat};
        \addplot [mark=square*, red] table[x=DimLinSys,y=ErrPGrad] {dat/ddr-quadrot/hexa_k2/data_rates.dat};
        \addplot [mark=square*, cyan] table[x=DimLinSys,y=ErrPGrad] {dat/ddr-quadrot/hexa_k3/data_rates.dat};
        \addplot [mark=square*, mark options=solid, black, dashed] table[x=DimLinSys,y=ErrPGrad] {dat/sddr-quadrot/hexa_k1/data_rates.dat};
        \addplot [mark=square*, mark options=solid, red, dashed] table[x=DimLinSys,y=ErrPGrad] {dat/sddr-quadrot/hexa_k2/data_rates.dat};
        \addplot [mark=square*, mark options=solid, cyan, dashed] table[x=DimLinSys,y=ErrPGrad] {dat/sddr-quadrot/hexa_k3/data_rates.dat};
      \end{loglogaxis}
    \end{tikzpicture}
    \subcaption{${\norm{\bvec{\Sigma},h}{\uGh{h} \underline{\varepsilon}_h}}$ and ${\norm{\bvec{\Sigma},h}{\suGh{h}\underline{\widehat \varepsilon}_h}}$}
  \end{minipage}
  \caption{Errors norm vs.\ linear system size using the standard (continuous lines) and serendipity spaces (dashed lines)  to solve the quad-rot problem of \cite[Section~5.2]{Di-Pietro:23} on the hexagonal mesh family.\label{fig:hexagonal.dof}}
\end{figure}


\section{A serendipity Stokes complex}\label{sec:stokes}

In this section we discuss a second application of the general construction considering the three-dimensional Stokes complex, another smoother variant of the three-dimensional de Rham complex.
Let $\Omega \subset \mathbb{R}^3$ be a polyhedral domain.
The Stokes complex reads:
\begin{equation} \label{eq:Stokescomplex}
  \begin{tikzpicture}[xscale=2, baseline={(H2.base)}]
    \node (H2) at (0,0) {$H^2(\Omega)$};
    \node (H1curl) at (1.5,0) {$\bvec{H}^1(\CURL;\Omega)$};
    \node (H1) at (3,0) {$\bvec{H}^1(\Omega)$};
    \node (L2) at (4.5,0) {$L^2(\Omega)$.};

    \draw[->,>=latex] (H2) -- (H1curl) node[midway, above, font=\footnotesize]{$\GRAD$};
    \draw[->,>=latex] (H1curl) -- (H1) node[midway, above, font=\footnotesize]{$\CURL$};
    \draw[->,>=latex] (H1) -- (L2) node[midway, above, font=\footnotesize]{$\DIV$};
  \end{tikzpicture}
\end{equation}
Diagram~\eqref{eq:abstract.diagram} specialized to the present case becomes
\begin{equation}\label{eq:construction:stokes}
  \resizebox{\textwidth}{!}{%
    \begin{tikzpicture}[xscale=2,yscale=2.5]
      \node at (1,0,-2) {Stokes:};
      \node (SGr) at (2,0,-2) {$\SGr{h}$};
      \node (SCr) at (4,0,-2) {$\SCr{h}$};
      \node (SDr) at (6,0,-1) {$\SDr{h}$};

    \node at (1,0,0) {SStokes:};
    \node (SSGr) at (2,0,0) {$\SSGr{h}$};
    \node (SSCr) at (4,0,0) {$\SSCr{h}$};

    \node at (1,2,0) {SDDR3d:};
    \node (SXGrad) at (2,2,0) {$\SXgrad{k}{h}$};
    \node (SXCurl) at (4,2,0) {$\SXcurl{h}$};
    \node (Xdiv) at (6,2,-1) {$\Xdiv{h}$};
    \node (Polyk) at (8,1,-1) {$\XL{k}{h}$};
    
    \node at (1,2,-2) {DDR3d:};    
    \node (XGrad) at (2,2,-2) {$\Xgrad{k-1,k-1,k}{h}$};
    \node (XCurl) at (4,2,-2) {$\Xcurl{k,k,k}{h}$};

    \draw [->,>=latex] (SSGr) -- (SSCr) node[midway, below,font=\footnotesize] {$\suGh{h}$};
    \draw [->,>=latex] (SSCr) -- (SDr) node[midway, below,font=\footnotesize] {$\suCh{h}$};
        
    \draw [->,>=latex] (SGr) -- (SCr) node[midway, above,font=\footnotesize] {$\uGh{h}$};
    \draw [->,>=latex] (SCr) -- (SDr) node[midway, above,font=\footnotesize] {$\uCh{h}$};
    \draw [->,>=latex] (SDr) -- (Polyk) node[midway, below,font=\footnotesize] {$\uDh{h}$};

    \draw [->,>=latex] (SXGrad) -- (SXCurl) node[midway, below,font=\footnotesize] {$\sGddr{h}$};
    \draw [->,>=latex] (SXCurl) -- (Xdiv) node[midway, below,font=\footnotesize] {$\sCddr{h}$};
        
    \draw [->,>=latex] (XGrad) -- (XCurl) node[midway, above,font=\footnotesize] {$\Gddr{h}$};
    \draw [->,>=latex] (XCurl) -- (Xdiv) node[midway, above,font=\footnotesize] {$\Cddr{h}$};
    \draw [->,>=latex] (Xdiv) -- (Polyk) node[midway, above,font=\footnotesize] {$\Dh{h}$};
    
    \draw [->,>=latex] (SSGr) to [bend right=20] node[midway, right, font=\footnotesize] {$\EV{h}$} (SGr) ;
    \draw [->,dashed] (SGr) to [bend right=20] node[midway, left, font=\footnotesize] {$\RV{h}$} (SSGr);
    \draw [->,>=latex] (SSCr) to [bend right=20] node[midway, right, font=\footnotesize] {$\uvec{E}_{V,\CURL,h}$} (SCr) ;
    \draw [->,dashed] (SCr) to [bend right=20] node[midway, left, font=\footnotesize] {$\suvec{R}_{V,\CURL,h}$} (SSCr);
        
    \draw [->,>=latex] (SXGrad) to [bend right=20] node[midway, right, font=\footnotesize] {$\Egrad{h}$} (XGrad);
    \draw [->,dashed] (XGrad) to [bend right=20] node[midway, left, font=\footnotesize] {$\Rgrad{h}$} (SXGrad);
    \draw [->,>=latex] (SXCurl) to [bend right=20] node[midway, right, font=\footnotesize] {$\Ecurl{h}$} (XCurl);
    \draw [->,dashed] (XCurl) to [bend right=20] node[midway, left, font=\footnotesize] {$\Rcurl{h}$} (SXCurl);
    
    \draw [->] (XCurl) to [bend left=20] node[pos=0.6, right,font=\footnotesize] {$\Injcurl$} (SCr);
    \draw [->,dashed] (SCr) to [bend left=20] node[pos=0.4, left,font=\footnotesize]  {$\Rescurl$} (XCurl);
    \draw [->] (XGrad) to [bend left=20] node[pos=0.6, right,font=\footnotesize] {$\Injgrad$} (SGr);
    \draw [->,dashed] (SGr) to [bend left=20] node[pos=0.4, left,font=\footnotesize]  {$\Resgrad$} (XGrad);
    \draw [->] (Xdiv) to [bend left=20] node[pos=0.6, right,font=\footnotesize] {$\Injdiv$} (SDr);
    \draw [->,dashed] (SDr) to [bend left=20] node[pos=0.4, left,font=\footnotesize]  {$\Resdiv$} (Xdiv);
    
    \draw [->] (SXCurl) to [bend left=20] node[pos=0.4, right,font=\footnotesize] {$\sInjcurl$} (SSCr);
    \draw [->,dashed] (SSCr) to [bend left=20] node[pos=0.6, left,font=\footnotesize]  {$\sRescurl$} (SXCurl);
    \draw [->] (SXGrad) to [bend left=20] node[pos=0.4, right,font=\footnotesize] {$\sInjgrad$} (SSGr);
    \draw [->,dashed] (SSGr) to [bend left=20] node[pos=0.6, left,font=\footnotesize]  {$\sResgrad$} (SXGrad);
    
  \end{tikzpicture}
  }
\end{equation}
The top horizontal portion of this diagram corresponds to \eqref{eq:ddr3d.sddr3d}.
In the rest of this section we will provide precise definitions of the remaining spaces and operators involved in the construction.

\subsection{Discrete Stokes complex}

We will start by giving a brief overview of the construction of a discrete counterpart of the complex \eqref{eq:Stokescomplex} developed in \cite{Hanot:23}.

\subsubsection{Discrete spaces}

For each edge $T_1 \in \mathcal{M}_{1,h}$, we will need the following space spanned by vector-valued polynomial functions that are normal to $T_1$:
\[
\TPoly{k}
\coloneq \left\lbrace
p_1 \bvec{n}_1 + p_2 \bvec{n}_2 \st p_1, p_2 \in \Poly{k}(T_1)
\right\rbrace,
\]
where $\bvec{n}_1$ and $\bvec{n}_2$ are two arbitrary orthogonal unit vectors normal to $T_1$.
The discrete counterparts of the spaces $H^2(\Omega)$, $\bvec{H}^1(\CURL;\Omega)$, $\bvec{H}^1(\Omega)$, and $L^2(\Omega)$ read:
\[
\SGr{h} \coloneq \Xgrad{k-1,k}{h} \times \SGr{\compl,h},\quad
\SCr{h} \coloneq \Xcurl{k,k}{h} \times \SCr{\compl,h},\quad
\SDr{h} \coloneq \Xdiv{h} \times \SDr{\compl,h}
\]
where the additional components with respect to the standard three-dimensional DDR spaces are given by
\begin{align*}
  \SGr{\compl,h} &\coloneq
  \bigtimes_{T_2\in\mathcal{M}_{2,h}}\Poly{k-1}(T_2) \times
  \bigtimes_{T_1\in\mathcal{M}_{1,h}}\TPoly{k}\times\Real^{3\mathcal{M}_{0,h}},
  \\
  \SCr{\compl,h} &\coloneq
  \begin{aligned}[t]
    &
      \bigtimes_{T_2\in\mathcal{M}_{2,h}}\left(
      \Poly{k-1}(T_2)\times\Goly{k}(T_2)\times\cGoly{k}(T_2)
      \right)
    \\
    &\times
    \bigtimes_{T_1\in\mathcal{M}_{1,h}}\left(
      \vPoly{k+1}(T_1;\Real^3)\times\TPoly{k}
    \right)\times\left(
    \Real^{3\mathcal{M}_{0,h}}
    \right)^2,
  \end{aligned}
  \\
  \SDr{\compl,h} &\coloneq 
    \bigtimes_{T_2\in\mathcal{M}_{2,h}}\left(
    \Goly{k}(T_2)\times\cGoly{k}(T_2)
    \right) \times \bigtimes_{T_1\in\mathcal{M}_{1,h}}\widetilde{\mathcal{P}}^{k+3}(T_1;\Real^3),
\end{align*}
where, to write $\SCr{h}$, we have decomposed the space $\widetilde{\mathcal{P}}^{k+2}(\mathcal{M}_{1,h};\Real^3)$ in \cite[Definition~(3.3)]{Hanot:23} as $\bigtimes_{T_1\in\mathcal{M}_{1,h}}\left(\Poly{k}(T_1)\times\TPoly{k}\right)\times\Real^{3\mathcal{M}_{0,h}}$ and $\widetilde{\mathcal{P}}^{m}(T_1;\Real^3)$ denotes the space of vector-valued functions over $T_1$ whose components are in $\lPoly{m}(T_1)$ and are continuous on $T_1$.

\subsubsection{Discrete gradient}

Let $\underline{q}_h=(\underline{q}_{w,h},\underline{q}_{\compl,h})\in\SGr{h}$ with

\[
\underline{q}_{\compl,h} \coloneq \left( (G_{q,T_2})_{T_2\in\mathcal{M}_{2,h}}, (\bvec{G}_{q,T_1})_{T_1\in\mathcal{M}_{1,h}}, (\bvec{G}_{q,T_0})_{T_0\in\mathcal{M}_{0,h}} \right)
\in \SGr{\compl,h},
\]
  where 
    $G_{q,T_2}$, $\bvec{G}_{q,T_1}$, and $\bvec{G}_{q,T_0}$
  have, respectively, the meaning of
  a normal gradient to the face $T_2$,
  a normal gradient to the edge $T_1$,
  and a full gradient at the vertex $T_0$.
The DDR discrete gradient is completed to map from $\SGr{h}$ to $\SCr{h}$ by adding the following component:
\[
\begin{aligned}
  \uvec{d}_{\GRAD,\compl,h}^k\underline{q}_{\compl,h}\coloneq \Big(
  &(G_{q,T_2}, \Gproj{k}{T_2} \bvec{RG}^k_{T_2} \underline{q}_{\compl,T_2}, \Gcproj{k}{T_2} \bvec{RG}^k_{T_2} \underline{q}_{\compl,T_2})_{T_2\in\mathcal{M}_{2,h}},
  \\
  &(\bvec{G}_{q,T_1},\bvec{v}_{T_1}'\times\bvec{t}_{T_1})_{T_1\in\mathcal{M}_{1,h}},
  \\
  &(\bvec{G}_{q,T_0},\bvec{0})_{T_0\in\mathcal{M}_{0,h}}
  \Big) \in \SCr{\compl,h},
\end{aligned}
\]
where $\underline{q}_{\compl,T_2}$ is the restriction of $\underline{q}_{\compl,h}$ to the elements neighbooring $T_2$, 
$\bvec{RG}^k_{T_2}$ is the rotor of the normal gradient defined by
\[
\int_{T_2} \bvec{RG}^k_{T_2} \underline{q}_{T_2} \cdot \bvec{w} = - \int_{T_2} G_{q,T_2} \ROT \bvec{w} - \sum_{T_1 \in \mathcal{M}_{1,T_2}} \omega_{T_2T_1} \int_{T_1} (\bvec{G}_{q,T_1} \cdot \normal_{T_2} ) (\bvec{w} \cdot \tangent_{T_1})
\quad \forall \bvec{w} \in \vPoly{k}(T_2),
\] 
and $\bvec{v}_{T_1}'$ is the derivative along the edge $T_1$ 
of the function $\bvec{v}_{T_1}$ such that $\bvec{\pi}^k_{\bvec{\mathcal{P}},T_1}\bvec{v}_{T_1} = \bvec{G}_{q,{T_1}}$
and for all $T_0\in\mathcal{M}_{0,T_1}$, $\bvec{v}_{T_1}(\bvec{x}_{T_0}) = \bvec{G}_{q,T_0}$.
The discrete gradient $\uGh{h} : \SGr{h} \to \SCr{h}$ is then given by
\begin{equation}\label{eq:uGh2}
  \uGh{h}~\underline{q}_h\coloneq\left(
  \Gddr{h}\underline{q}_{w,h},\uvec{d}_{\GRAD,\compl,h}^k\underline{q}_{\compl,h}
  \right).  
\end{equation}

\subsubsection{Discrete curl}

For $\uvec{v}_h=(\uvec{v}_{w,h},\uvec{v}_{\compl,h})\in\SCr{h}$, the component $\uvec{v}_{\compl,h}$ is given by
\[
\uvec{v}_{\compl,h}\coloneq\left((v_{T_2},\bvec{R}_{v,\cvec{G},{T_2}},\bvec{R}^c_{v,\cvec{G},T_2})_{T_2\in\mathcal{M}_{2,h}}, (\bvec{R}_{v,T_1},\bvec{v}_{n,T_1})_{T_1\in\mathcal{M}_{1,h}}, (\bvec{v}_{T_0}, \bvec{R}_{v,T_0})_{T_0\in\mathcal{M}_{0,h}} \right)
\in \SCr{\compl,h},
\]
where $v_{T_2}$, $(\bvec{R}_{v,\cvec{G},{T_2}},\bvec{R}^c_{v,\cvec{G},T_2})$, $\bvec{R}_{v,T_1}$, and $(v_{T_0}\bvec{R}_{v,T_0})$ have, respectively, the meaning of the normal flux accross the face $T_2$, 
the normal gradient of the tangential components to the face $T_2$, 
the tangential component of the curl plus the normal gradient of the tangential component to the edge $T_1$, and the value of the function and of its curl at the vertex $T_0$.
  The discrete curl in the DDR complex \eqref{eq:glob.3d-ddr.complex} is completed by adding the following component in order to obtain a map from $\SCr{h}$ to $\SDr{h}$:
\[
\begin{aligned}
  \uvec{d}_{\CURL,\compl,h}^k\uvec{v}_{\compl,h}\coloneq
  \Big(
  &(\Gproj{k}{T_2} \bvec{C}^k_{T_2} \uvec{v}_{\compl,T_2},\bvec{R}_{v,\cvec{G},T_2}, \Gcproj{k}{T_2}\bvec{C}^k_{T_2} \uvec{v}_{\compl,T_2},\bvec{R}^c_{v,\cvec{G},T_2})_{T_2\in\mathcal{M}_{2,h}},
  \\
  &(\bvec{C}^k_{T_1} \uvec{v}_{\compl,T_1})_{T_1\in\mathcal{M}_{1,h}}
  \Big) \in \SDr{\compl,h},
\end{aligned}
\]
where $\uvec{v}_{\compl,T_2}$ is the restriction of $\uvec{v}_{\compl,h}$ to the elements sharing $T_2$,
$\uvec{v}_{\compl,T_1}$ the restriction of $\uvec{v}_{\compl,h}$ to the elements sharing $T_1$, $\bvec{C}^k_{T_2} $ is the face curl defined in \eqref{eq:vCF}, and $\bvec{C}^k_{T_1}$ is such that $\bvec{C}^k_{T_1} \uvec{v}_{\compl,T_1} (\bvec{x}_{T_0})= \bvec{R}_{v,T_0}$ and $\vlproj{k+1}{T_1}\bvec{C}^k_{T_1} \uvec{v}_{\compl,T_1}= \bvec{R}_{v,T_1}- \bvec{v}_{\bvec{n},T_1}' ~\times ~{\tangent}_{T_1}$,
with $\bvec{v}_{\bvec{n},T_1}$ such that 
$\vlproj{k}{T_1}\bvec{v}_{\bvec{n},T_1} = \bvec{v}_{n,T_1}$
and for all $T_0\in\mathcal{M}_{0,T_1}$, $\bvec{v}_{\bvec{n},T_1}(\bvec{x}_{T_0}) = \bvec{v}_{T_0}$.
The discrete curl is then given by 
\begin{align}\label{eq:uCh}
  \uCh{h}~\uvec{v}_h&\coloneq\left(
  \Cddr{h}\uvec{v}_{w,h},\uvec{d}_{\CURL,\compl,h}^k\uvec{v}_{\compl,h}
  \right).
\end{align}

\subsubsection{Discrete divergence}

The discrete divergence is nothing but the original DDR divergence defined by \eqref{eq:ddr.3d:operators} but with domain $\SDr{h}$ instead of $\Xdiv{h}$:
For all $\uvec{w}_h=(\uvec{w}_{w,h},\uvec{w}_{\compl,h})\in\SDr{h}$,
\[
  \uDh{h}~\uvec{w}_h \coloneq \Dh{h}\uvec{w}_{w,h}.
\]

\subsubsection{Discrete Stokes complex}

The discrete counterpart of the Stokes complex \eqref{eq:Stokescomplex} %
  which appears at the bottom and back of diagram \eqref{eq:construction:stokes}
is given by:
\begin{equation}\label{eq:d.stokes.complex}  
  \begin{tikzpicture}[xscale=2, baseline={(H2.base)}]
    \node at (-1,0) {Stokes:};
    
    \node (H2) at (0,0) {$\SGr{h}$};
    \node (H1curl) at (1.5,0) {$\SCr{h}$};
    \node (H1) at (3,0) {$\SDr{h}$};
    \node (L2) at (4.5,0) {$\XL{k}{h}$.};

    \draw[->,>=latex] (H2) -- (H1curl) node[midway, above, font=\footnotesize]{$\uGh{h}$};
    \draw[->,>=latex] (H1curl) -- (H1) node[midway, above, font=\footnotesize]{$\uCh{h}$};
    \draw[->,>=latex] (H1) -- (L2) node[midway, above, font=\footnotesize]{$\uDh{h}$};
  \end{tikzpicture}
\end{equation}

\subsection{Extension and reduction maps between the three-dimensional DDR and Stokes complexes}

We next define extension and reduction operators between the three-dimensional DDR complex \eqref{eq:glob.3d-ddr.complex} and the discrete Stokes complex \eqref{eq:d.stokes.complex} that satisfy Assumption \ref{Isom.W.V}.
The proof is similar to that of Theorem \ref{thm.hom.rotrot} and is omitted for the sake of brevity.
It follows once again from Remark~\ref{rem:iso.cohom} that \eqref{eq:glob.3d-ddr.complex} and \eqref{eq:d.stokes.complex} have isomorphic cohomologies.

The extension operators are such that:
For all $\underline{q}_{w,h}\in\Xgrad{k-1,k-1,k}{h}$,
all $\uvec{v}_{w,h}\in\Xcurl{k,k,k}{h}$,
and all $\uvec{w}_{w,h}\in\Xdiv{h}$,
\[
\Injgrad \underline{q}_{w,h} \coloneq \big( \underline{q}_{w,h},\underline{0}\big),\quad
\Injcurl\uvec{v}_{w,h} \coloneq \big( \uvec{v}_{w,h},\uvec{0}\big),\quad
\Injdiv \uvec{w}_{w,h} \coloneq \big( \uvec{w}_{w,h},\uvec{0}\big).
\]
The reduction map is such that, for all $\underline{q}_h=(\underline{q}_{w,h},\underline{q}_{\compl,h})\in\SGr{h}$,
all $\uvec{v}_h=(\uvec{v}_{w,h},\uvec{v}_{\compl,h})\in\SCr{h}$,
and all $\uvec{w}_h=(\uvec{w}_{w,h},\uvec{w}_{\compl,h})\in\SDr{h}$,
\[
\Resgrad\underline{q}_h \coloneq \underline{q}_{w,h},\quad
\Rescurl\uvec{v}_h \coloneq \uvec{v}_{w,h},\quad
\Resdiv\uvec {w}_h \coloneq \uvec{w}_{w,h}.
\]
  For future reference, we note the following isomorphisms, which are a direct consequence of the above definitions:
  \begin{equation}\label{eq:isomorphism.Res.grad.curl}
    \text{$\Kernel \Resgrad \cong \SGr{\compl,h}$ and $\Kernel \Rescurl \cong \SCr{\compl,h}$}.
  \end{equation}

\subsection{Serendipity Stokes complex and homological properties}\label{sec:discrete.stokescomplex:serendipity}

Applying the construction of Section~\ref{sec:gen} to the Stokes complex and recalling the isomorphisms \eqref{eq:isomorphism.Res.grad.curl}, we obtain the following serendipity version of the spaces $\SGr{h}$ and $\SCr{h}$:
\begin{equation}\label{eq:S.S.Spaces}
  \SSGr{h} \coloneq \SXgrad{k}{h} \times \SGr{\compl,h},\quad
  \SSCr{h} \coloneq \SXcurl{h} \times \SCr{\compl,h},  
\end{equation}
where $\SXgrad{k}{h}$ and $\SXcurl{h}$ are the serendipity DDR spaces defined by \eqref{eq:SXgrad.SXcurl}.

We write generic elements $\widehat{\underline{q}}_h$ of $\SSGr{h}$ and $\suvec{v}_h$ of $\SSCr{h}$ 
respectively as $\widehat{\underline{q}}_h=(\widehat{\underline{q}}_{w,h},{\underline{q}}_{\compl,h})$
and $\suvec{v}_h=(\suvec{v}_{w,h},\uvec{v}_{\compl,h})$ 
with $\widehat{\underline{q}}_{w,h}\in\SXgrad{k}{h}$, 
$\suvec{v}_{w,h}\in\SXcurl{h}$,
and ${\underline{q}}_{\compl,h} \in \SGr{\compl,h}$ 
and $\uvec{v}_{\compl,h} \in \SCr{\compl,h}$. 

According to \eqref{eq:def.sExt}, we define the extensions of the SDDR spaces into serendipity Stokes spaces as follows:
For all $\widehat{\underline{q}}_{w,h} \in \SXgrad{k}{h}$ and all $\suvec{v}_{w,h} \in \SXcurl{h}$,
\[
  \text{
  $\sInjgrad\widehat{\underline{q}}_{w,h} \coloneq \big( \widehat{\underline{q}}_{w,h}, {\underline{0}} \big)$
  and
  $\sInjcurl\suvec{v}_{w,h} \coloneq \big( \suvec{v}_{w,h}, \uvec{0} \big)$.
  }
\]
The reduction map between the SStokes and the SDDR complexes is given by \eqref{eq:def.sRed}:
  For all $(\widehat{\underline{q}}_{w,h}, {\underline{q}}_{\compl,h}) \in \SSGr{h}$
  and all $(\suvec{v}_{w,h},\uvec{v}_{\compl,h}) \in \SSCr{h}$,
\[
\text{
$\sResgrad(\widehat{\underline{q}}_{w,h},{\underline{q}}_{\compl,h})
\coloneq \widehat{\underline{q}}_{w,h}$
and
$\sRescurl(\suvec{v}_{w,h},\uvec{v}_{\compl,h})
\coloneq \suvec{v}_{w,h}$.
}
\]

  By \eqref{eq:def.ReV}, the reduction map from the Stokes to the SStokes complexes are given by:
For all $\underline{q}_h = (\underline{q}_{w,h}, \underline{q}_{\compl,h})\in\SV{h}$ and
all $\uvec{v}_h = (\uvec{v}_{w,h}, \uvec{v}_{\compl,h}) \in \SSigma{h}$,
\[
  \text{
  $\widehat{\underline{R}}_{V,\GRAD,h}\underline{q}_h
  \coloneq
  \Big(\Rgrad{h}\Resgrad\underline{q}_h,
  \underline{q}_{\compl,h}
    \Big)$
    and
  $\suvec{R}_{V,\CURL,h}\uvec{v}_h
  \coloneq
  \Big(
  \Rcurl{h}\Rescurl\uvec{v}_h,
  \uvec{v}_{\compl,h}
  \Big)$.
  }
\]
The extension operators 
from the SStokes to the Stokes complexes are defined according to \eqref{eq:def.EX}:
For all $\widehat{\underline{q}}_h\in\SSGr{h}$ and all $\suvec{v}_h\in\SSCr{h}$,
\[
\begin{aligned}
  \underline{E}_{V,\GRAD,h}\widehat{\underline{q}}_h
  &\coloneq
  \Injgrad\Egrad{h} \widehat{\underline{q}}_{w,h}+(\widehat{\underline{0}},{\underline{q}}_{\compl,h}),
  \\ 
  \uvec{E}_{V,\CURL,h}\suvec{v}_h
  &\coloneq
  \Injcurl\Ecurl{h}\suvec{v}_{w,h} + (\suvec{0},\uvec{v}_{\compl,h}).
\end{aligned}
\]

Using \eqref{eq:def.d}, the serendipity discrete differential operators are such that, for all $(\widehat{\underline{q}}_h,\suvec{v}_h)\in\SSGr{h}\times\SSCr{h}$,
\begin{equation*}
  \begin{aligned}
    \suGh{h} \widehat{\underline{q}}_h&\coloneq\big(\sGddr{h}\widehat{\underline{q}}_h,\uGh{h} (\widehat{\underline{0}},{\underline{q}}_{\compl,h})\big)\overset{\eqref{eq:uGh2},\eqref{eq:S.S.Spaces}}=\big(\sGddr{h}\widehat{\underline{q}}_h,\uvec{d}_{\GRAD,\compl,h}^k {\underline{q}}_{\compl,h}\big),\\
    \suCh{h}\suvec{v}_h&\coloneq \big(\sCddr{h}\suvec{v}_{w,h},\uCh{h}(\suvec{0},\uvec{v}_{\compl,h})\big)\overset{\eqref{eq:uCh},\eqref{eq:S.S.Spaces}}=\big(\sCddr{h}\suvec{v}_{w,h},\uvec{d}_{\CURL,\compl,h}^k \uvec{v}_{\compl,h}\big).
  \end{aligned}
\end{equation*}
  This completes the definition of the serendipity Stokes complex corresponding to the bottom front complex in diagram \eqref{eq:construction:stokes}.
The following theorem can be proved using arguments similar to Theorem~\ref{thm:homological.properties}.
The details are omitted for the sake of brevity.

\begin{theorem}[Homological properties of the complexes in \eqref{eq:construction:stokes}]\label{thm:homological.properties2}  
    All the complexes in the diagram \eqref{eq:construction:stokes} have cohomologies that are isomorphic to the cohomology of the continuous de Rham complex.
\end{theorem}

\section*{Acknowledgements}

Funded by the European Union (ERC Synergy, NEMESIS, project number 101115663).
Views and opinions expressed are however those of the authors only and do not necessarily reflect those of the European Union or the European Research Council Executive Agency. Neither the European Union nor the granting authority can be held responsible for them.


\printbibliography

\end{document}